\documentclass[11pt,onecolumn,draftcls]{IEEEtran}

\usepackage{verbatim}
\usepackage[]{float,latexsym}
\usepackage{amsfonts,amsmath,mathrsfs,amssymb,amsbsy}
\usepackage[final]{graphicx}
\usepackage{times,cite}
\usepackage{subfig}
\usepackage{url}

\newtheorem{theorem}{Theorem}[section]
\newtheorem{corrly}{Corollary}

\newtheorem{lemma}{Lemma}

\newtheorem{remark}{Remark}

\newtheorem{algorithm}{Algorithm}

\newcommand{\Prob}{\mathbb{P}}
\newcommand{\Expect}{\mathbb{E}}
\newcommand{\indic}{\mathbb{I}}
\newcommand{\defeq}{\overset{\Delta}{=}}

\newtheorem{problem}{Problem}

\usepackage{color,soul}
\definecolor{lightblue}{rgb}{.7, .8, 1}
\definecolor{lightgreen}{rgb}{.6, 1, .6}
\usepackage{color}
\definecolor{brown}{rgb}{1,0.38,0.03}

\definecolor{OliveGreen}{rgb}{.2,0.6,0.2}
\definecolor{BrickRed}{rgb}{.7,0.2,0.2}

\newcommand{\ignore}[1]{} 

\long\def\symbolfootnote[#1]#2{\begingroup%
\def\thefootnote{\fnsymbol{footnote}}\footnote[#1]{#2}\endgroup}

\newcommand{\ADD}{{\mathsf{ADD}}}
\newcommand{\CADD}{{\mathsf{CADD}}}

\newcommand{\WADD}{{\mathsf{WADD}}}
\newcommand{\FAR}{{\mathsf{FAR}}}

\newcommand{\PFA}{\mathsf{PFA}}

\newcommand{\ANO}{{\mathsf{ANO}}}
\newcommand{\PDC}{{\mathsf{PDC}}}

\newcommand{\tauc}{\tau_{\scriptscriptstyle \text{C}}}

\newcommand{\taud}{\tau_{\scriptscriptstyle \text{D}}}
\newcommand{\tauw}{\tau_{\scriptscriptstyle \text{W}}}
\newcommand{\lambdad}{\lambda_{\scriptscriptstyle \text{D}}}
\newcommand{\Lambdad}{\Lambda_{\scriptscriptstyle \text{D}}}

\newcommand{\Psic}{\Psi_{\scriptscriptstyle \text{C}}}
\newcommand{\Psiw}{\Psi_{\scriptscriptstyle \text{W}}}
\newcommand{\Nd}{N_{\scriptscriptstyle \text{D}}}
\newcommand{\lambdaInf}{\lambda_{\scriptscriptstyle \infty}}
\newcommand{\LambdaInf}{\Lambda_{\scriptscriptstyle \infty}}

\begin{document}
\title{Data-Efficient Quickest Change Detection in Minimax Settings}
\author{Taposh Banerjee and Venugopal V. Veeravalli\\ ECE Department and Coordinated Science Laboratory\\
1308 West Main Street, Urbana, IL~~61801\\Email: {banerje5, vvv}$@$illinois.edu. Tel: +1(217) 333-0144, Fax: +1(217) 244-1642.}

\maketitle

\vspace{-1.5cm}

\symbolfootnote[0]{\small Preliminary version of this paper has been presented at ICASSP 2012.  This research was supported in part by the National Science Foundation under grant CCF 08-30169 and DMS 12-22498, through the University of Illinois at 
Urbana-Champaign. This research was also supported in part by the U.S. Defense Threat Reduction Agency through subcontract 147755 at the University of Illinois from prime award HDTRA1-10-1-0086.}

\begin{abstract}
The classical problem of quickest change detection is studied with an additional
constraint on the cost of observations used in the detection process.
The change point is modeled as an unknown constant, and minimax formulations
are proposed for the problem.
The objective in these formulations is to find a stopping time and an
on-off observation control policy for the observation sequence, to minimize
a version of the worst possible average delay, subject to constraints on the
false alarm rate and the fraction of time observations are taken before change.
An algorithm called DE-CuSum is proposed and
is shown to be asymptotically optimal for the proposed formulations, as
the false alarm rate goes to zero.
Numerical results are used to show that the DE-CuSum algorithm has good trade-off curves and
performs significantly better than the approach of fractional sampling,
in which the observations are skipped using the outcome of a sequence of
coin tosses, independent of the observation process.
This work is guided by the insights gained
from an earlier study of a Bayesian version of this problem.
\end{abstract}


\section{Introduction} \label{sec:Intro}
In the problem of quickest change detection, a decision
maker observes a sequence of random variables $\{X_n\}$. At some point of
time $\gamma$, called the change point, the distribution of the random variables changes.
The goal of the decision maker is to
find a stopping time $\tau$ on the $\{X_n\}$, 
so as to minimize the average
value of the delay $\max\{0, \tau-\gamma\}$. The delay is zero
on the event $\{\tau<\gamma\}$, but this event is treated as a false alarm and is not
desirable. Thus, the average delay has to be minimized subject to a constraint
on the false alarm rate.
This problem finds
application in statistical quality control in industrial processes,
surveillance using sensor networks and cognitive radio networks;
see \cite{poor-hadj-qcd-book-2009}, \cite{tart-niki-bass-2013},
\cite{veer-bane-elsevierbook-2013}.

In the i.i.d. model of the quickest change detection problem,
the random variables
$\{X_n\}$ for $n<\gamma$ are independent and identically distributed
(i.i.d.) with probability density function (p.d.f) $f_0$, and $\{X_n\}$
for $n\geq \gamma$ are i.i.d. with p.d.f. $f_1$. In the Bayesian version of
the quickest change detection problem the change point $\gamma$ is modeled as
a random variable $\Gamma$.

In \cite{shir-siamtpa-1963}, \cite{tart-veer-siamtpa-2005} the i.i.d. model is studied in a Bayesian setting by
assuming the change point $\Gamma$ to be a geometrically distributed random
variable. The objective is to minimize the \textit{average detection delay}
with a constraint on the \textit{probability of false alarm}.
It is shown that under very
general conditions on $f_0$ and $f_1$, the
optimal stopping time is the one that stops the first time the \textit{a posteriori}
probability $\Prob(\Gamma\leq n | X_1,\cdots,X_n)$ crosses a pre-designed threshold.
The threshold is chosen to meet the false alarm constraint with equality. In the following we
refer to this algorithm as the Shiryaev algorithm.

In \cite{lord-amstat-1971}, \cite{poll-astat-1985}, \cite{mous-astat-1986}, \cite{ritov-astat-1990},
\cite{lai-ieeetit-1998}, \cite{tart-poll-polu-arxiv-2011}, no prior knowledge about
the distribution on the change point is assumed, and
the change point is modeled as an unknown constant.
In this non-Bayesian setting, the quickest change detection
problem is studied in two different minimax settings introduced in \cite{lord-amstat-1971} and \cite{poll-astat-1985}.
The objective in \cite{lord-amstat-1971}\;--\cite{tart-poll-polu-arxiv-2011}
is to minimize some version of the \textit{worst case average delay},
subject to a constraint on the \textit{mean time to false alarm}.
The results from these papers show that, variants of the Shiryaev-Roberts algorithm \cite{robe-technometrics-1966},
the latter being derived from the Shiryaev algorithm by setting the geometric parameter to zero,
and the CuSum algorithm \cite{page-biometrica-1954},
are asymptotically optimal for both the minimax formulations, as the mean time to false alarm goes to
infinity.

In many applications of quickest change detection, changes are infrequent and
there is a cost associated with acquiring observations (data). As a result, it is of interest to
study the classical quickest change detection problem
with an additional constraint on the
cost of observations used before the change point, with the cost of taking observations
after the change point being penalized through the metric on delay. In the following,
we refer to this problem as data-efficient quickest change detection.

In \cite{bane-veer-sqa-2012}, we studied data-efficient quickest change detection
in a Bayesian setting by adding another constraint to the Bayesian
formulation of \cite{shir-siamtpa-1963}.
The objective was to find a stopping time and an on-off observation control policy
on the observation sequence,
to minimize the average detection delay subject to constraints on the probability of false
alarm and the average number of observations used before the change point.
Thus unlike the classical quickest change detection problem, where the decision
maker only chooses one of the two controls, \textit{to stop and declare change} or \textit{to continue taking observations},
in the data-efficient quickest change detection problem we considered in \cite{bane-veer-sqa-2012},
the decision maker must also decide -- when the decision is to \textit{continue} --
whether to \textit{take} or \textit{skip} the next observation.

For the i.i.d. model, and for
geometrically distributed $\Gamma$, we showed in \cite{bane-veer-sqa-2012} that a two-threshold
algorithm is asymptotically optimal, as the probability of false alarm goes to zero.
This two-threshold algorithm, that we call the DE-Shiryaev algorithm in the following, is a generalized
version of the Shiryaev algorithm from \cite{shir-siamtpa-1963}.
In the DE-Shiryaev algorithm, the \textit{a posteriori} probability
that the change has already happened conditioned on available information, is computed at
each time step, and the change is declared the first time this probability crosses a threshold $A$.
When the \textit{a posteriori} probability is below this threshold $A$, observations are taken only
when this probability is above another threshold $B<A$.
When an observation is skipped, the \textit{a posteriori} probability is updated using the prior on the
change point random variable.
We also showed that, for reasonable values
of the false alarm constraint and the observation cost constraint, these two thresholds can be
selected independent of each other: the upper threshold $A$ can be selected directly from the false
alarm constraint and the lower threshold $B$ can be selected directly from the observation cost
constraint. Finally, we showed that the DE-Shiryaev algorithm achieves a significant
gain in performance over the approach of \textit{fractional sampling}, where the Shiryaev
algorithm is used and an observation is skipped based on the outcome of a coin toss.

In this paper we study the
data-efficient quickest change detection problem in a non-Bayesian setting,
by introducing an additional constraint on the cost of observations used
in the detection process, in the minimax settings
of \cite{lord-amstat-1971} and \cite{poll-astat-1985}. We first use the insights from
the Bayesian analysis in \cite{bane-veer-sqa-2012} to propose a metric for
data efficiency in the absence of knowledge of the distribution on the change point.
This metric is the fraction of time samples are taken before change. 
We then propose extensions of the minimax formulations in \cite{lord-amstat-1971}
and \cite{poll-astat-1985} by introducing an additional constraint on data efficiency
in these formulations. Thus, the objective is to find a stopping
time and an on-off observation control policy to minimize a version of the worst case average delay,
subject to constraints on the mean time to false alarm and the fraction of time observations
are taken before change. Then, motivated by the structure of the DE-Shiryaev algorithm,
we propose an extension of the CuSum algorithm from \cite{page-biometrica-1954}.
We call this extension the DE-CuSum algorithm.
We show that the DE-CuSum algorithm
inherits the good properties of the DE-Shiryaev algorithm. That is, the DE-CuSum algorithm is
asymptotically optimal, is easy to design, and provides substantial performance improvements 
over the approach of fractional sampling, where the CuSum algorithm is used and observations 
are skipped based on the outcome of a sequence of coin tosses, independent of the 
observations process. 

The problem of detecting an anomaly in the behavior of an industrial process, under cost
considerations, is also considered in the literature of statistical process control. There
it is studied under the heading of sampling rate control or sampling size control;
see \cite{taga-jqt-1998} and \cite{stou-etal-jamstaa-2000} for a detailed survey,
and the references in \cite{bane-veer-sqa-2012} for some recent results.
However, none of these references study the data-efficient quickest change detection
problem under the classical quickest change detection setting, as done by us in
\cite{bane-veer-sqa-2012} and in this paper. For a result similar to our work
in \cite{bane-veer-sqa-2012} in a Bayesian setting
see \cite{prem-kuma-infocom-2008}. See \cite{geng-etal-samplingrights-allerton-2012} 
and \cite{kris-ieeetit-whento-2012} for other interesting formulations of quickest change 
detection with observation control. 

Since our work in this paper on data-efficient non-Bayesian quickest change detection
is motivated by our work on data-efficient Bayesian quickest change detection \cite{bane-veer-sqa-2012},
in Section~\ref{sec:Bayesian}, we provide a detailed overview of the results from \cite{bane-veer-sqa-2012}.
We also comment on the insights provided by the Bayesian analysis,
which we use in the development of a theory
for the non-Bayesian setting. In Section~\ref{sec:MinimaxFormulation},
Section~\ref{sec:DECUSUMALGO}, and Section~\ref{sec:PerfAnalyDECuSum}, we provide details
of the minimax formulations, a description of the DE-CuSum algorithm and the analysis of the
DE-CuSum algorithm, respectively. We provide the numerical results in Section~\ref{sec:Trade-off}.

Table~\ref{Tab:Glossary} provides a glossary of the terms used in the paper.

\begin{table} [htbp]
 \centering
 \centering
 {\small
 \caption{Glossary}
\label{Tab:Glossary}
\begin{tabular}{|l|l|}
\hline
\cline{1-2}
\textbf{Symbol}&\textbf{Definition/Interpretation}\\
\hline
$o(1)$&$x=o(1)$ as $c \to c_0$, if $\forall \epsilon>0$, \\
&$\exists \delta>0$ s.t., $|x|\leq \epsilon$ if $|c-c_0|<\delta$\\
$O(1)$&$x=O(1)$ as $c \to c_0$, if $\exists \epsilon>0, \delta>0$ \\
& s.t., $|x|\leq \epsilon$ if $|c-c_0|<\delta$\\
$g(c)\sim h(c)$ & $\lim_{c \to c_0} \frac{g(c)}{h(c)}=1$ \\
\hspace{0.5cm} as $c \to c_0$ & or $g(c) =  h(c) (1+ o(1))$ as $c \to c_0$\\
$\Prob_{n}$ ($\Expect_{n}$)& Probability measure (expectation) \\
&when the change occurs at time $n$\\
$\Prob_\infty$ ($\Expect_\infty$) & Probability measure (expectation) \\
&when the change does not occur\\
$\text{ess sup}\ X$ & $\inf\{K \in \mathbb{R}: \Prob(X>K)=0\}$\\
$D (f_1\; \| \; f_0)$& K-L Divergence between $f_1$ and $f_0$, \\
&defined as $\Expect_1\left(\log \frac{f_1 (X)}{f_0(X)} \right)$\\
$D (f_0\; \| \; f_1)$& K-L Divergence between $f_0$ and $f_1$, \\
&defined as $\Expect_\infty\left(\log \frac{f_1 (X)}{f_0(X)} \right)$\\
$(x)^+$ & $\max\{x, 0\}$\\
$(x)^{h+}$ & $\max\{x, -h\}$\\
$M_n$ & $M_n=1$ if $X_n$ is used for decision making\\
$\Psi$&Policy for data-efficient quickest \\
& change detection $\{\tau, M_1, \cdots, M_\tau\}$\\
$\ADD(\Psi)$ & $\sum_{n=0}^\infty \Prob(\Gamma=n) \ \Expect_n \left[(\tau - \Gamma)^+\right]$ \\
$\PFA(\Psi)$ & $\sum_{n=0}^\infty \ \Prob(\Gamma=n) \Prob_n (\tau<\Gamma)$\\
$\FAR(\Psi)$ & $\frac{1}{\Expect_\infty[\tau]}$\\
$\WADD(\Psi)$ & $\underset{n \geq 1}{\operatorname{\sup}}\ \text{ess sup} \ \Expect_n\left[(\tau-n)^+| I_{n-1}\right]$\\
$\CADD(\Psi)$ & $\underset{n \geq 1}{\operatorname{\sup}}\ \Expect_n[\tau-n| \tau\geq n]$\\
$\PDC(\Psi)$ & $\limsup_{n}  \frac{1}{n} \Expect_n \left[\sum_{k=1}^{n-1} M_k \Big| \tau \geq n\right]$\\
\hline
\end{tabular}}
 \end{table}

\section{Data-Efficient Bayesian Quickest Change Detection}
\label{sec:Bayesian}
In this section we review the Bayesian version of the data-efficient quickest change
detection we studied in \cite{bane-veer-sqa-2012}. We consider the i.i.d. model, i.e.,
$\{X_n\}$ is a sequence of random variables, $\{X_n\}$ for $n<\Gamma$ are i.i.d. with p.d.f. $f_0$,
and $\{X_n\}$ for $n\geq \Gamma$ are i.i.d. with p.d.f. $f_1$. We further assume that $\Gamma$
is geometrically distributed with parameter $\rho$:
\[\Prob(\Gamma=n) = (1-\rho)^{n-1} \rho.\]

For data-efficient quickest change detection we consider the following class of control policies.
At each time $n, n\geq 0$, a decision is made as to whether to \textit{take} or \textit{skip} the observation at time $n+1$.
Let $M_n$ be the indicator random variable such that $M_n=1$ if $X_n$ is used for decision making,
and $M_n=0$ otherwise. Thus, $M_{n+1}$ is a function of the information available at time $n$, i.e.,
\[
M_{n+1} = \phi_n(I_n),
\]
where, $\phi_n$ is the control law at time $n$, and
\[
I_n = \left[ M_1, \ldots, M_n, X_1^{(M_1)}, \ldots, X_n^{(M_n)} \right],
\]
represents the information at time $n$.
Here, $X_i^{(M_i)}$ represents $X_i$ if $M_i=1$, otherwise $X_i$ is absent from the information vector $I_n$.
Also, $I_0$ is an empty set.

For time $n\geq1$, based on the information vector $I_n$, a decision is made whether to \textit{stop and declare change}
or \textit{to continue taking observations}. Let $\tau$ be a stopping time on the information sequence $\{I_n\}$, that is
$\indic_{\{\tau=n\}}$ is a measurable function of $I_n$. Here, $\indic_F$ represents the indicator of the event $F$.
Thus, a policy for data-efficient quickest change detection is $\Psi = \{\tau, \phi_0, \ldots, \phi_{\tau-1} \}$.

Define the average detection delay
\[
\ADD(\Psi) \stackrel{\Delta}{=} \Expect\left[(\tau - \Gamma)^+\right],
\]
the probability of false alarm
\[
\PFA(\Psi) \stackrel{\Delta}{=} \Prob(\tau<\Gamma),
\]
and the metric for data-efficiency in the Bayesian setting we considered in \cite{bane-veer-sqa-2012},
the average number of observations
used before the change point,
\[
\ANO(\Psi) \stackrel{\Delta}{=} \Expect\left[\sum_{n=1}^{\min(\tau,\Gamma-1)} M_n\right].
\]
  The objective in \cite{bane-veer-sqa-2012} is to solve the following optimization problem:
\begin{problem}\label{prob:DEBayes}
\begin{eqnarray}
\label{eq:basicproblem}
\underset{\Psi}{\text{minimize}} && \ADD(\Psi) , \nonumber\\
\text{subject to } &&\PFA(\Psi) \leq \alpha,  \\
\text{  and  } &&\ANO(\Psi) \leq \zeta. \nonumber
\end{eqnarray}
Here, $\alpha$ and $\zeta$ are given constraints.
\end{problem}
\begin{remark}
When $\zeta \geq \Expect[\Gamma]-1$, Problem~\ref{prob:DEBayes} reduces to the classical Bayesian quickest change detection problem considered by Shiryaev in \cite{shir-siamtpa-1963}.
\end{remark}
\subsection{The DE-Shiryaev algorithm}\label{sec:DEShiryaevAlgo}
Define,
\[
p_n = \Prob \left( \Gamma \leq n \; | \; I_n \right).
\]
Then, the two-threshold algorithm from \cite{bane-veer-sqa-2012} is:
\begin{algorithm}[DE-Shiryaev: $\Psi(A,B)$]
\label{algo:TwoThreshold}
Start with $p_0=0$ and use the following control, with $B<A$, for $n\geq 0$:
\begin{equation}
\label{eq:TwoThresholdAlgo}
\begin{split}
               M_{n+1} & = \phi_n (p_n) =  \begin{cases}
               0 & \mbox{ if } p_n < B\\
               1 & \mbox{ if } p_n \geq B
               \end{cases}\\
               \taud &= \inf\left\{ n\geq 1: p_n > A \right\}.
               \end{split}
\end{equation}
The probability $p_n$ is updated using the following recursions:
\[
p_{n+1} = \begin{cases}
\tilde{p}_n  & \text{~if~} M_{n+1} = 0\\
\frac{\tilde{p}_n L(X_{n+1})}{ \tilde{p}_n  L(X_{n+1}) + (1-\tilde{p}_n )} & \text{~if~} M_{n+1} = 1
\end{cases}
\]
with $\tilde{p}_n = p_n + (1-p_n) \rho$ and
$L(X_{n+1}) = f_1(X_{n+1})/f_0(X_{n+1})$.
\end{algorithm}

\begin{remark}
With $B=0$ the DE-Shiryaev algorithm reduces to the Shiryaev algorithm from \cite{shir-siamtpa-1963}.
\end{remark}

The motivation for this algorithm comes from the fact that $p_n$ is a sufficient statistics
for a Lagrangian relaxation of Problem~\ref{prob:DEBayes}. This relaxed problem can be studied
using dynamic programming, and numerical studies of the resulting Bellman equation
shows that the DE-Shiryaev algorithm is optimal for a wide choice of system parameters.
For an analytical
justification see Section~\ref{sec:SQAOpt} below.

When Algorithm \ref{algo:TwoThreshold} is employed, the probability $p_n$ typically evolves as depicted in Fig. \ref{fig:Pkevolution}.
\begin{figure}[htb]
\center
\includegraphics[width=8cm, height=5cm]{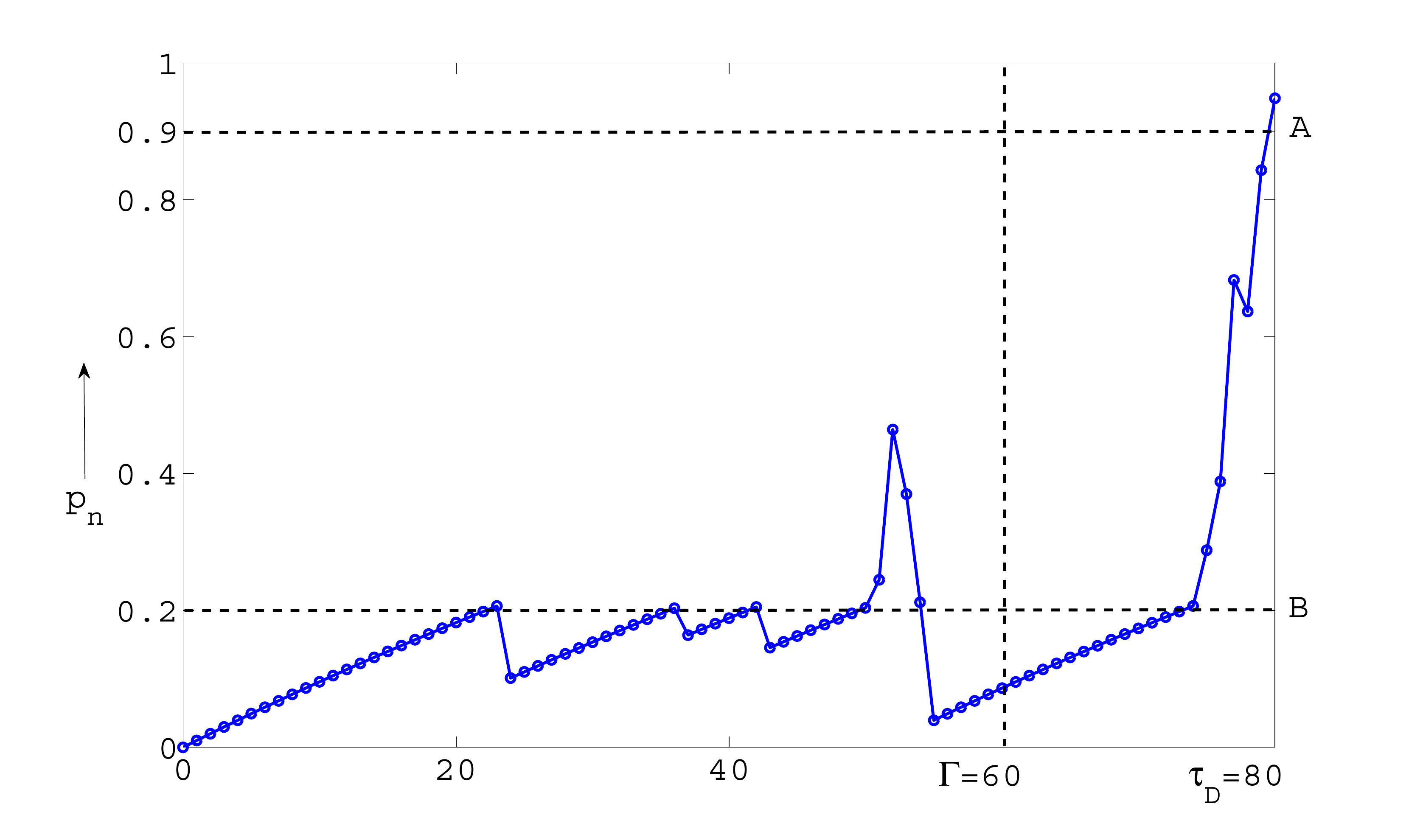}
\caption{Typical evolution of $p_n$ for $f_0 \sim {\cal N}(0,1)$, $f_1 \sim {\cal N}(0.8,1)$, and $\rho=0.01$, with thresholds $A=0.9$ and $B=0.2$. }
\label{fig:Pkevolution}
\end{figure}
As observed in Fig.~\ref{fig:Pkevolution}, the evolution starts with an initial value of $p_0=0$. This is because
we have implicitly assumed that the probability that the change has already happened even before we start taking
observations is zero. Also, note that when $p_n < B$, $p_n$ increases monotonically.
This is because when an observation is skipped, $p_n$ is updated using the prior on the change point,
and as a result
the probability that the change
has already happened increases monotonically. The change is declared at time $\taud$, the first time $p_n$
crosses the threshold $A$.

\subsection{Asymptotic Optimality and trade-off curves}\label{sec:SQAOpt}
It is shown in \cite{bane-veer-sqa-2012} that the $\PFA$ and $\ADD$ of the DE-Shiryaev algorithm approach that of
the Shiryaev algorithm as $\alpha \to 0$. Specifically, the following theorem is proved.
\smallskip
\begin{theorem}\label{thm:SQA}
If
\[
0< D(f_0 \;||\; f_1) < \infty \ \ \ \mbox{ and } \  \ \ 0< D(f_1\; ||\; f_0) < \infty,
\]
and
$L(X)$ is non-arithmetic (see \cite{sieg-seq-anal-book-1985}), then for a fixed $\zeta$,
the threshold $B$ can be selected such that for every $A>B$,
\[
\ANO(\Psi(A,B)) \leq \zeta,
\]
and with $B$ fixed to this value,
\begin{equation}\label{eq:BayesianDelayOpt}
\ADD(\Psi(A,B)) \sim \frac{|\log(\alpha)|}{D(f_1\;||\;f_0) + |\log(1-\rho)|} \mbox{ as } \alpha \to 0.
\end{equation}
and
\begin{equation}\label{eq:BayesianPFAOpt}
\PFA(\Psi(A,B)) \sim \alpha \left( \int_{0}^\infty e^{-x} dR(x)\right) \mbox{ as } \alpha \to 0.
\end{equation}
Here, $R(x)$ is the asymptotic overshoot distribution of the random walk
$\sum_{k=1}^n (L(X_k)+|\log(1-\rho)|)$, when it crosses a large positive boundary under $f_1$.
Since, \eqref{eq:BayesianDelayOpt} and \eqref{eq:BayesianPFAOpt} are also the performance of the Shiryaev algorithm
as $\alpha\to 0$ \cite{tart-veer-siamtpa-2005}, the DE-Shiryaev algorithm is asymptotically optimal.
\end{theorem}
\smallskip
\begin{remark}
Equation \eqref{eq:BayesianPFAOpt} shows that $\PFA$ is not a function of the threshold $B$.
In \cite{bane-veer-sqa-2012}, it is shown that as $\alpha\to 0$ and as $\rho \to 0$, $\ANO$ is a function of $B$ alone.
Thus, for reasonable values of the constraints $\alpha$ and $\beta$, the constraints can be met independent of each other.
\end{remark}
\smallskip
\begin{remark}
The statement of Theorem \ref{thm:SQA} is stronger than the claim that the DE-Shiryaev algorithm is
asymptotically optimal. This is true because
\[
\PFA(\Psi(A,B)) = \Expect[1-p_{\taud}] \leq 1-A.
\]
Thus, with $A=1-\alpha$, $\PFA(\Psi(A,B))\leq \alpha$, and with $B$ chosen as mentioned in the theorem,
\eqref{eq:BayesianDelayOpt} alone establishes the asymptotic optimality of the DE-Shiryaev
algorithm.
\end{remark}
\smallskip
\begin{remark}
Although \eqref{eq:BayesianDelayOpt} is true for each fixed value of $\zeta$, as $\zeta$ becomes smaller,
a much smaller value of $\alpha$ is needed before the asymptotics `kick in'.
\end{remark}
\smallskip

Fig. \ref{fig:TradeoffDEShiryaev} compares the performance of the Shiryaev algorithm, the DE-Shiryaev algorithm
and the fractional sampling scheme, for $\zeta=50$.
In the fractional sampling scheme, the Shiryaev algorithm is used and samples are skipped
by tossing a biased coin (with probability of success 50/99), without looking at the state of the system.
When a sample is skipped in the fractional sampling scheme,
the Shiryaev statistic is updated
using the prior on change point.
The figure clearly shows a substantial gap in performance between
the DE-Shiryaev algorithm and the fractional sampling scheme.
\begin{figure}[htb]
\center
\includegraphics[width=8cm, height=5cm]{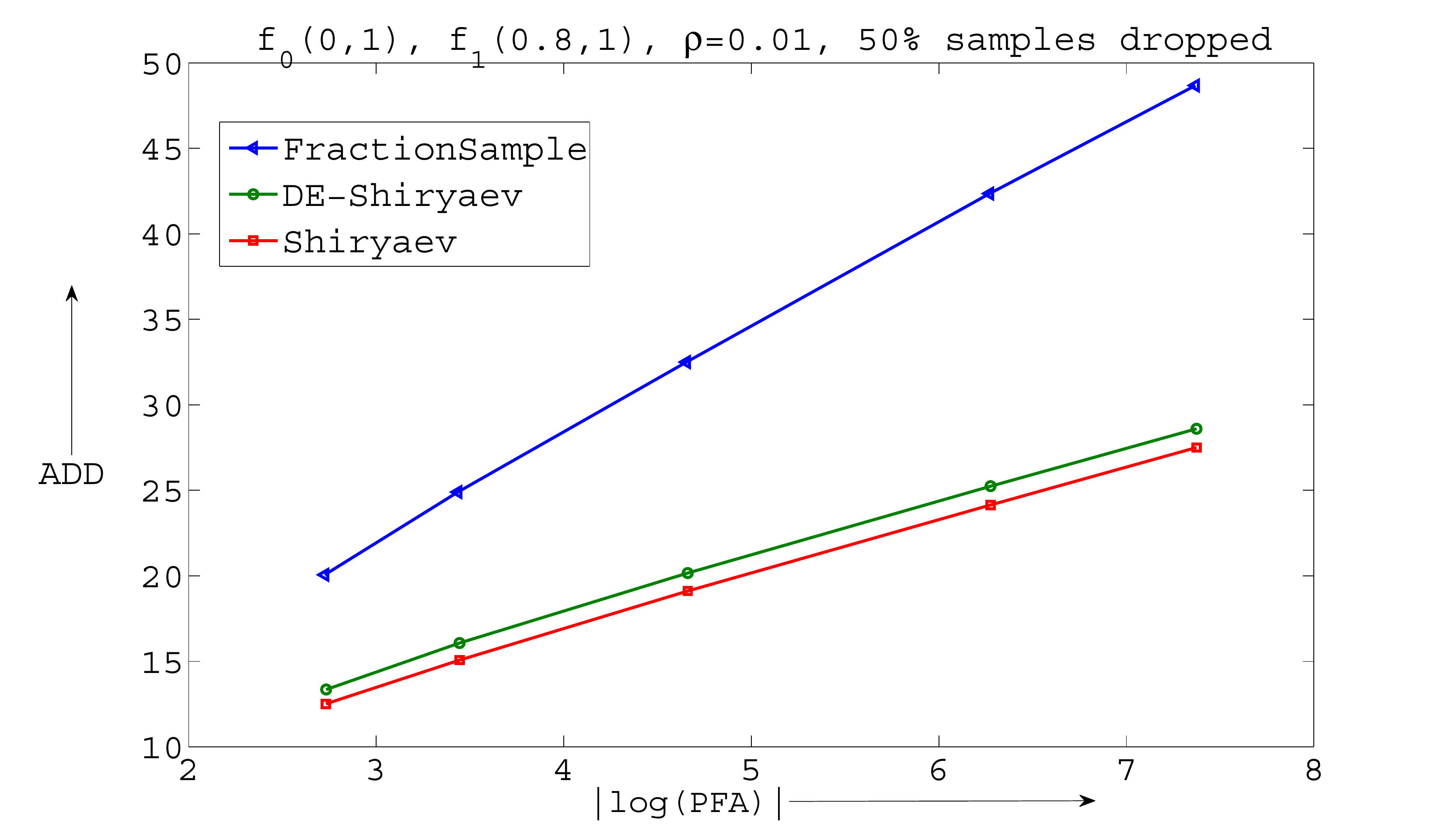}
\caption{Comparative performance of schemes for $f_0 \sim {\cal N}(0,1)$, $f_1 \sim {\cal N}(0.8,1)$, and $\rho=0.01$. }
\label{fig:TradeoffDEShiryaev}
\end{figure}

More accurate estimates of the delay and that of
$\ANO$ are available in \cite{bane-veer-sqa-2012}.

\subsection{Insights from the Bayesian setting}
\label{sec:insightsFromSQA}
We make the following observations on the evolution of the statistic $p_n$ in Fig. \ref{fig:Pkevolution}.
\begin{enumerate}
\item Let
\[
t(B) = \inf\{n\geq 1: p_n > B\}.
\]
Then after $t(B)$, the number of samples skipped when $p_n$ goes below $B$ is a function of the undershoot of $p_n$
and the geometric parameter $\rho$.
If $L^*(X_n)$ is defined as
\[
L^*(X_n) = \begin{cases}
L(X_n)  & \text{~if~} M_{n} = 1\\
1 & \text{~if~} M_{n} = 0
\end{cases}.
\]
Then $\frac{p_n}{1-p_n}$ can be shown to be equal to
\[\frac{p_n}{1-p_n} = \frac{\sum_{k=1}^n (1-\rho)^{k-1} \rho \prod_{i=k}^n L^*(X_i) }{\Prob(\Gamma>n)}.\]
Thus $\frac{p_n}{1-p_n}$ is the average likelihood ratio of all the
observations taken till time $n$, and since there is a one-to-one mapping between $p_n$
and $\frac{p_n}{1-p_n}$, we see that the number of samples skipped is a function of the likelihood ratio
of the observations taken.
\item When $p_n$ crosses $B$ from below, it does so with an overshoot that is bounded by $\rho$. This is because
\[
p_{n+1} - p_n = (1-p_n)\rho \leq \rho.
\]
For small values of $\rho$, this overshoot is essentially zero, and the evolution of $p_n$ is roughly statistically independent
of its past evolution. Thus, beyond $t(B)$, the evolution of $p_n$ can be seen as a sequence of \textit{two-sided} statistically
independent tests, each two-sided test being a test for sequential hypothesis testing between ``$H_0 = \text{pre-change}$",
and ``$H_1 = \text{post-change}$". If the decision in the two-sided test is $H_0$, then samples are skipped depending
on the likelihood ratio of the observations, and the two-sided test is repeated on the samples beyond the skipped samples.
The change is declared the first time the decision in a two-sided test is $H_1$.
%
%
\item Because of the above interpretation of the evolution of the DE-Shiryaev algorithm as a sequence of roughly independent
two-sided tests, we see that the constraint on the observation cost is met
by delaying the measurement process on the basis of the prior statistical knowledge of the
change point, and then beyond $t(B)$, controlling the fraction of time $p_n$ is above $B$, i.e.,
controlling the fraction of time samples are taken.
\end{enumerate}
These insights will be crucial to the development of the theory for data-efficient quickest change detection
in the non-Bayesian setting.

\section{Data-Efficient Minimax Formulation}
\label{sec:MinimaxFormulation}

In the absence of a prior knowledge on the distribution of the change point,
as is standard in classical quickest change detection literature, we model the change point as
an unknown constant $\gamma$. As a result, the quantities $\ADD,\PFA,\ANO$
in Problem~\ref{prob:DEBayes} are not well defined. Thus, we study the data-efficient
quickest change detection problem in a minimax setting.
In this paper we consider two most popular minimax formulations: one is due to Pollak \cite{poll-astat-1985}
and another is due to Lorden \cite{lord-amstat-1971}.

We will use the insights from the Bayesian setting of Section~\ref{sec:Bayesian} to study data-efficient minimax quickest change detection.
Our development will essentially follow the layout of the Bayesian setting. Specifically, we first
propose two minimax formulations for data-efficient
quickest change detection. Motivated by the structure of the DE-Shiryaev algorithm,
we then propose an algorithm for data-efficient quickest change detection in the minimax settings.
This algorithm is a generalized version of the CuSum algorithm \cite{page-biometrica-1954}.
We call this algorithm the DE-CuSum algorithm.
We show that the DE-CuSum algorithm is asymptotically optimal under both minimax settings.
We also show that in the DE-CuSum algorithm,
the constraints on false alarm and observation cost can be met independent of each other.
Finally, we show that we can achieve a substantial
gain in performance by using the DE-CuSum algorithm as compared to the approach of fractional sampling.

We first propose a metric for data-efficiency in a non-Bayesian setting.
In Section~\ref{sec:insightsFromSQA}, we saw that in the DE-Shiryaev algorithm, observation cost constraint is
met using an initial wait, and by controlling the fraction of time observations are taken, after the initial wait.
In the absence of
prior statistical knowledge on the change point such an initial wait cannot be justified.
This motivates us to seek control policies that can meet a
constraint on the fraction of time observations are taken before change.
With $M_n$, $I_n$, $\tau$, and $\Psi$ as defined earlier in Section~\ref{sec:Bayesian},
we propose the following duty cycle based observation cost metric, Pre-change Duty Cycle
($\PDC$):
\begin{equation}
\label{eq:PDC}
\PDC(\Psi) = \limsup_{n}  \frac{1}{n} \Expect_n \left[\sum_{k=1}^{n-1} M_k \Big| \tau \geq n\right].
\end{equation}
Clearly, $\PDC\leq 1$.

We now discuss why we use $\limsup$ rather than $\sup$ in defining $\PDC$.
In all reasonable policies $\Psi$, $M_1$ will typically be set to 1.
As mentioned earlier, this is because an initial wait cannot be justified without a prior
statistical knowledge of the change point. As a result, in \eqref{eq:PDC},
we cannot replace the $\limsup$ by $\sup$, because the latter would give us a $\PDC$ value of 1.
Even otherwise, without any prior knowledge on the change point, it is reasonable to
assume that the value of $\gamma$ is large, and hence the $\PDC$ metric defined
in \eqref{eq:PDC} is a reasonable metric for our problem.

For false alarm, we consider the metric used in \cite{lord-amstat-1971} and \cite{poll-astat-1985},
the mean time to false alarm or its reciprocal, the false alarm rate:
\begin{equation}\label{eq:FAR_Def}
\FAR(\Psi) = \frac{1}{\Expect_\infty \left[ \tau\right]}.
\end{equation}

For delay we consider two possibilities:
the minimax setting of Pollak \cite{poll-astat-1985} where
the delay metric is the following supremum over time of the conditional delay\footnote{We are only interested
in those policies for which the $\CADD$ is well defined.}
\begin{equation}\label{eq:CADD_Def}
\CADD(\Psi) = \sup_n \ \ \Expect_n \left[ \tau-n | \tau \geq n \right],
\end{equation}
or the minimax setting of Lorden \cite{lord-amstat-1971}, where the
delay metric is the supremum over time of the essential supremum
of the conditional delay
\begin{equation}\label{eq:WADD_Def}
\WADD(\Psi) = \sup_n  \text{ess sup}\; \Expect_n \left[ (\tau-n)^+ | I_{n-1} \right].
\end{equation}
%
Note that unlike the delay metric in \cite{lord-amstat-1971},
$\WADD$ in \eqref{eq:WADD_Def} is a function of the observation control through
$I_{n-1} = \left[ M_1, \ldots, M_{n-1}, X_1^{(M_1)}, \ldots, X_{n-1}^{(M_{n-1})} \right]$,
which may not contain the entire set of observations.

Since, $\{\tau=n\}$ belongs to the sigma algebra generated by $I_{n-1}$, we have
\[\CADD(\Psi) \leq \WADD(\Psi).\]
\smallskip
%

Our first minimax formulation is the following data-efficient extension of Pollak \cite{poll-astat-1985}
\smallskip
\begin{problem}\label{prob:DEPollak}
\begin{eqnarray}
\label{eq:MinimaxProblemPollak}
\underset{\Psi}{\text{minimize}} && \CADD(\Psi),\nonumber \\
\text{subject to } && \FAR(\Psi) \leq \alpha, \\
\text{  and  } &&\PDC(\Psi) \leq \beta. \nonumber
\end{eqnarray}
Here, $0 \leq \alpha, \beta \leq 1$ are given constraints.
\end{problem}
\smallskip
We are also interested in the data-efficient extension of the minimax formulation of Lorden \cite{lord-amstat-1971}.
\begin{problem}\label{prob:DELorden}
\begin{eqnarray}
\label{eq:MinimaxProblemLorden}
\underset{\Psi}{\text{minimize}} && \WADD(\Psi),\nonumber \\
\text{subject to } && \FAR(\Psi) \leq \alpha, \\
\text{  and  } &&\PDC(\Psi) \leq \beta. \nonumber
\end{eqnarray}
Here, $0 \leq \alpha, \beta \leq 1$ are given constraints.
\end{problem}
\smallskip

\begin{remark}
With $\beta=1$, Problem~\ref{prob:DEPollak} reduces to the minimax formulation of Pollak in \cite{poll-astat-1985},
and Problem~\ref{prob:DELorden} reduces to the minimax formulation of Lorden in \cite{lord-amstat-1971}.
\end{remark}

In \cite{page-biometrica-1954}, the following algorithm called the CuSum algorithm is proposed:
\smallskip
\begin{algorithm}[$\mathrm{CuSum}$: $\Psic$]\label{algo:CuSum}
Start with $C_0=0$, and update the statistic $C_n$ as
\[C_{n+1} = \left(C_n + \log L(X_{n+1})\right)^+, \]
where $(x)^+ = \max\{0,x\}$. Stop at
\[\tauc = \inf\{n\geq1: C_n > D\}.\]
\end{algorithm}
It is shown by Lai in \cite{lai-ieeetit-1998} that the CuSum algorithm is asymptotically optimal for
both Problem~\ref{prob:DEPollak} and Problem~\ref{prob:DELorden}, with $\beta=1$, as $\alpha \to 0$
(see Section~\ref{sec:AnalyCuSum} for a precise statement).

In the following we propose the DE-CuSum algorithm,
an extension of the CuSum algorithm for the data-efficient setting,
and show that it is asymptotically optimal, for each fixed $\beta$, as $\alpha\to0$;
see Section~\ref{sec:AsymptoticOpt}.

\section{The DE-CuSum algorithm}
We now present the DE-CuSum algorithm.
\label{sec:DECUSUMALGO}
\begin{algorithm}[$\mathrm{DE-CuSum}$: $\Psiw(D, \mu, h)$]
\label{algo:DECuSum}
Start with $W_0=0$ and fix $\mu>0$, $D>0$ and $h\geq0$. For $n\geq 0$ use the following control:
\begin{equation*}
\label{eq:DECuSum}
\begin{split}
               M_{n+1} &= \begin{cases}
               0 & \mbox{ if \ } W_n < 0\\
               1 & \mbox{ if \ } W_n \geq 0
               \end{cases}\\
               \tauw &= \inf\left\{ n\geq 1: W_n > D \right\}.
               \end{split}
\end{equation*}
The statistic $W_n$ is updated using the following recursions:
\[
W_{n+1} = \begin{cases}
\min\{W_n + \mu, 0\} & \text{~if~} M_{n+1} = 0\\
(W_n + \log L(X_{n+1}))^{h+} & \text{~if~} M_{n+1} = 1
\end{cases}
\]
where $(x)^{h+} = \max\{x, -h\}$.
\end{algorithm}


When $h=\infty$, the DE-CuSum algorithm works as follows. The statistic $W_n$ starts at 0,
and evolves according to the CuSum algorithm till it goes below 0.
When $W_n$ goes below 0, it does so with an undershoot. Beyond this,
$W_n$ is incremented deterministically (by using the recursion $W_{n+1}=W_n+\mu$),
and observations are skipped till $W_n$
crosses 0 from below.
As a consequence, the number of observations that are skipped
is determined by the undershoot (log likelihood ratio of the observations)
as well as the parameter $\mu$.
When $W_n$ crosses 0 from below, it is reset to 0. Once $W_n=0$,
the process renews itself and continues to evolve this way until $W_n > D$, at which
time a change is declared.

If $h< \infty$, $W_n$ is truncated to $-h$ when $W_n$
goes below 0 from above. In other
words, the undershoot is reset to $-h$ if its magnitude is larger than $h$.
A finite value of $h$ guarantees
that the number of samples skipped is bounded by $\frac{h}{\mu}+1$.
This feature will be crucial to the $\WADD$ analysis of the DE-CuSum algorithm in
Section~\ref{sec:DECuSumCADDWADD}.

If $h=0$, the DE-CuSum statistic $W_n$ never becomes negative and hence
reduces to the CuSum statistic and evolves as: $W_0=0$, and for $n\geq 0$,
\[W_{n+1} = \max\{0, W_n + \log L(X_{n+1})\}.\]

Thus, $\mu$ is a substitute for the Bayesian prior $\rho$
that is used in the DE-Shiryaev algorithm described in Section~\ref{sec:DEShiryaevAlgo}.
But unlike $\rho$ which represents
a prior statistical knowledge of the change point, $\mu$ is a design parameter. An appropriate value of
$\mu$ is selected to meet the constraint on $\PDC$; see Section~\ref{sec:PDCConstraint} for details.

The evolution of the DE-CuSum algorithm is plotted in Fig.~\ref{fig:DECuSum_evolution}.
\begin{figure}[htb]
\center
\includegraphics[width=8cm, height=5cm]{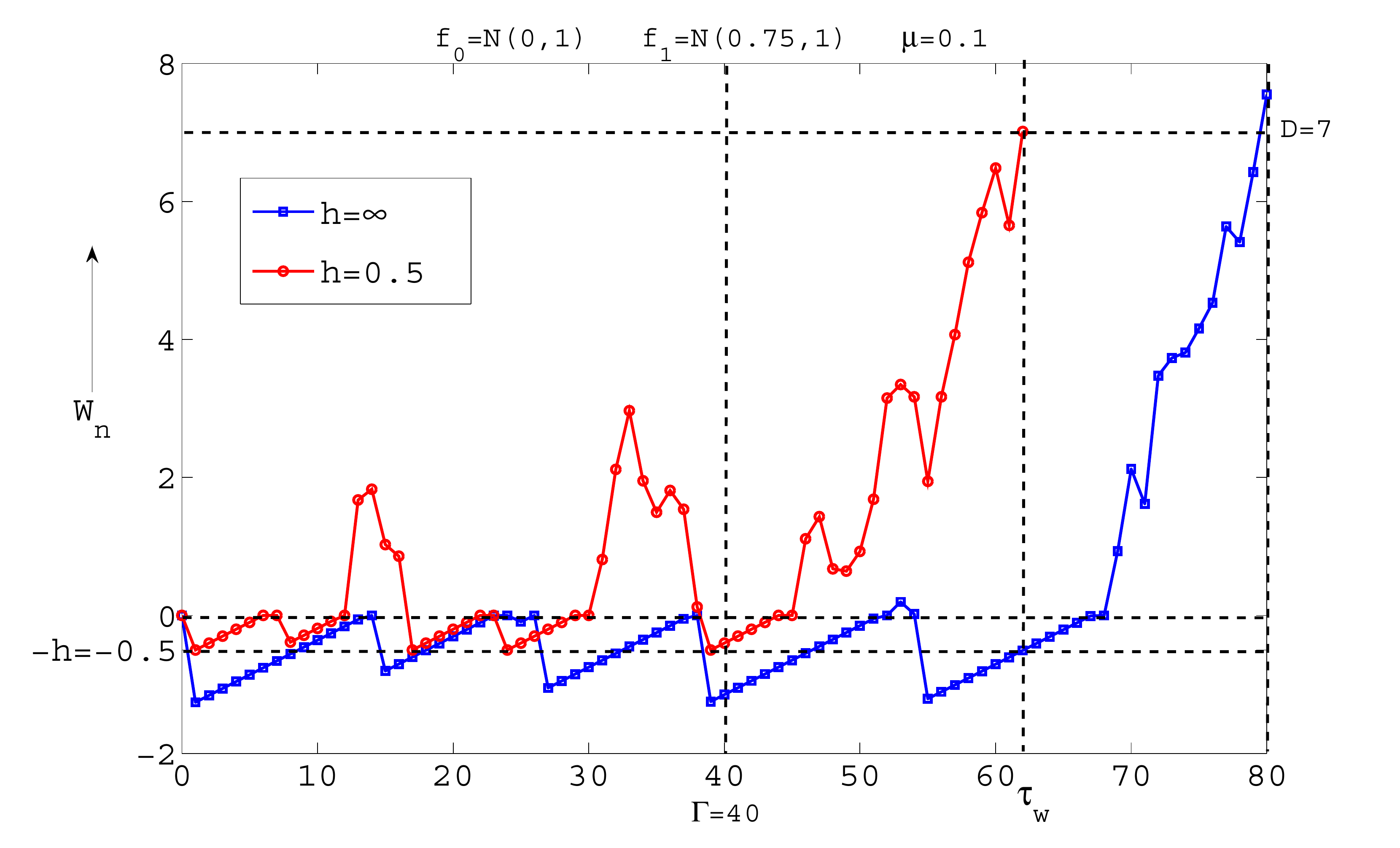}
\caption{Typical evolution of $W_n$ for $f_0 \sim {\cal N}(0,1)$, $f_1 \sim {\cal N}(0.75,1)$, $\Gamma=40$, $D=7$,
$\mu=0.1$, with two different values of $h$: $h=\infty$ and $h=0.5$. When $h=0.5$, 
the undershoots are truncated at $-0.5$.}
\label{fig:DECuSum_evolution}
\end{figure}
In analogy with the evolution of the DE-Shiryaev algorithm,
the DE-CuSum algorithm can also be seen as a sequence of independent two-sided tests.
In each two-sided test a Sequential Probability Ratio Test (SPRT) \cite{wald-wolf-amstat-1948},
with log boundaries $D$ and $0$, is used to distinguish between the two hypotheses ``$H_0 = \text{pre-change}$"
and ``$H_1 = \text{post-change}$".
If the decision in the SPRT is in favor of $H_0$, then samples are skipped based on the
likelihood ratio of all the observations taken in the SPRT. A change is declared
the first time the decision in the sequence of SPRTs is in favor of $H_1$.
If $h=0$, no samples are skipped and the DE-CuSum reduces to the CuSum algorithm, i.e., to a sequence of SPRTs
(also see \cite{sieg-seq-anal-book-1985}).

Unless it is required to have a bound on the maximum number of samples skipped, the DE-CuSum
algorithm can be controlled by just two-parameters: $D$ and $\mu$. We will show in the following
that these two parameters can be selected independent of each other directly from the constraints. That
is the threshold $D$ can be selected so that $\FAR \leq \alpha$ independent of the value of $\mu$.
Also, it is possible to select a value of $\mu$ such that $\PDC \leq \beta$ independent of the choice of $D$.

\begin{remark}
With the way the DE-CuSum algorithm is defined, we will see in the following that it may not be possible to
meet $\PDC$ constraints that are close to 1, with equality. We ignore this issue in the rest of the paper,
as in many practical settings the preferred value of $\PDC$ would be closer to 0 than 1.
But, we remark that the DE-CuSum algorithm can be easily modified to achieve $\PDC$ values
that are close to 1 by resetting $W_n$ to zero if the undershoot is smaller
than a pre-designed threshold.
\end{remark}

\begin{remark}
One can also modify the Shiryaev-Roberts algorithm \cite{robe-technometrics-1966}
and obtain a two-threshold version of it,
with an upper threshold used for stopping and a lower threshold used for on-off observation control.
Also note that the SPRTs of the two-sides tests considered above have a lower threshold of 0.
One can also propose variants of the DE-CuSum algorithm with a negative lower threshold for the SPRTs.
\end{remark}

\begin{remark}
For the CuSum algorithm, the supremum in \eqref{eq:CADD_Def} and \eqref{eq:WADD_Def}
 is achieved when the change is applied at time $n=1$ (see also \eqref{eq:CuSumDelayCADDWADD}).
 This is useful from the point of view of
simulating the test. However,
in the data-efficient setting, since the information vector also contains
information about missed samples, 
the worst case change point in \eqref{eq:CADD_Def} would depend on the observation control and
may not be $n=1$. But note that
in the DE-CuSum algorithm, the test statistic evolves
as a Markov process. As a result, the worst case usually occurs in the initial slots,
before the process hits stationarity. This is useful from the point of view of simulating
the algorithm.
In the analysis of the DE-CuSum algorithm provided in Section~\ref{sec:PerfAnalyDECuSum} below,
we will see that the $\WADD$ of the DE-CuSum algorithm is equal to its delay when change occurs
at $n=1$, plus a constant.
Similarly, even if computing the $\PDC$
may be a bit difficult using simulations, we will provide simple numerically-computable
upper bound on the $\PDC$ of the DE-CuSum algorithm
that can be used to ensure that the $\PDC$ constraint is satisfied.
\end{remark}

\section{Analysis and design of the DE-CuSum algorithm}
\label{sec:PerfAnalyDECuSum}
The identification/intepretation of the DE-CuSum algorithm as a sequence of two-sided tests will now
be used in this section to perform its asymptotic analysis.

Recall that the DE-CuSum algorithm can be seen as a sequence of two sided tests,
each two-sided test contains
an SPRT and a possible sojourn below zero. The length of the latter being dependent
on the likelihood ratio of the observations.

Define the following two functions:
\[
\Phi(W_k) = W_k + \log L(X_{k+1}),
\]
and
\[
\bar{\Phi}(W_k) = W_k + \mu.
\]
Using these functions we define the stopping time for an SPRT
\begin{equation}
\label{eq:lambda}
\lambdad  \defeq \inf\{n\geq 1: \Phi(W_{n-1}) \notin [0,D], \ W_0=0 \}.
\end{equation}
At the stopping time $\lambdad$ for the SPRT, if the statistic $W_{\lambdad}=x<0$,
then the time spent below zero is equal to
$T(x,0)$, where for $x<y$
\begin{equation}
\label{eq:TXY}
T(x,y,\mu)  \defeq \inf\{n\geq 1: \bar{\Phi}(W_{n-1}) > y, W_0=x\},
\end{equation}
with $T(0,0,\mu)=0$. Note that
\begin{equation}
\label{eq:TXYCeil}
T(x,y,\mu) = \lceil (y-x)/\mu\rceil.
\end{equation}
We also define the stopping time for the two-sided test
\begin{equation}
\label{eq:Lambda}
\Lambdad = \lambdad + T((W_{\lambdad})^{h+},0,\mu) \; \indic_{\{W_{\lambdad} < 0\}}.
\end{equation}
\smallskip
Let $\lambdaInf$ be the variable $\lambdad$ when the threshold $D=\infty$.

\begin{figure}[htb]
\center
\includegraphics[width=8cm, height=5cm]{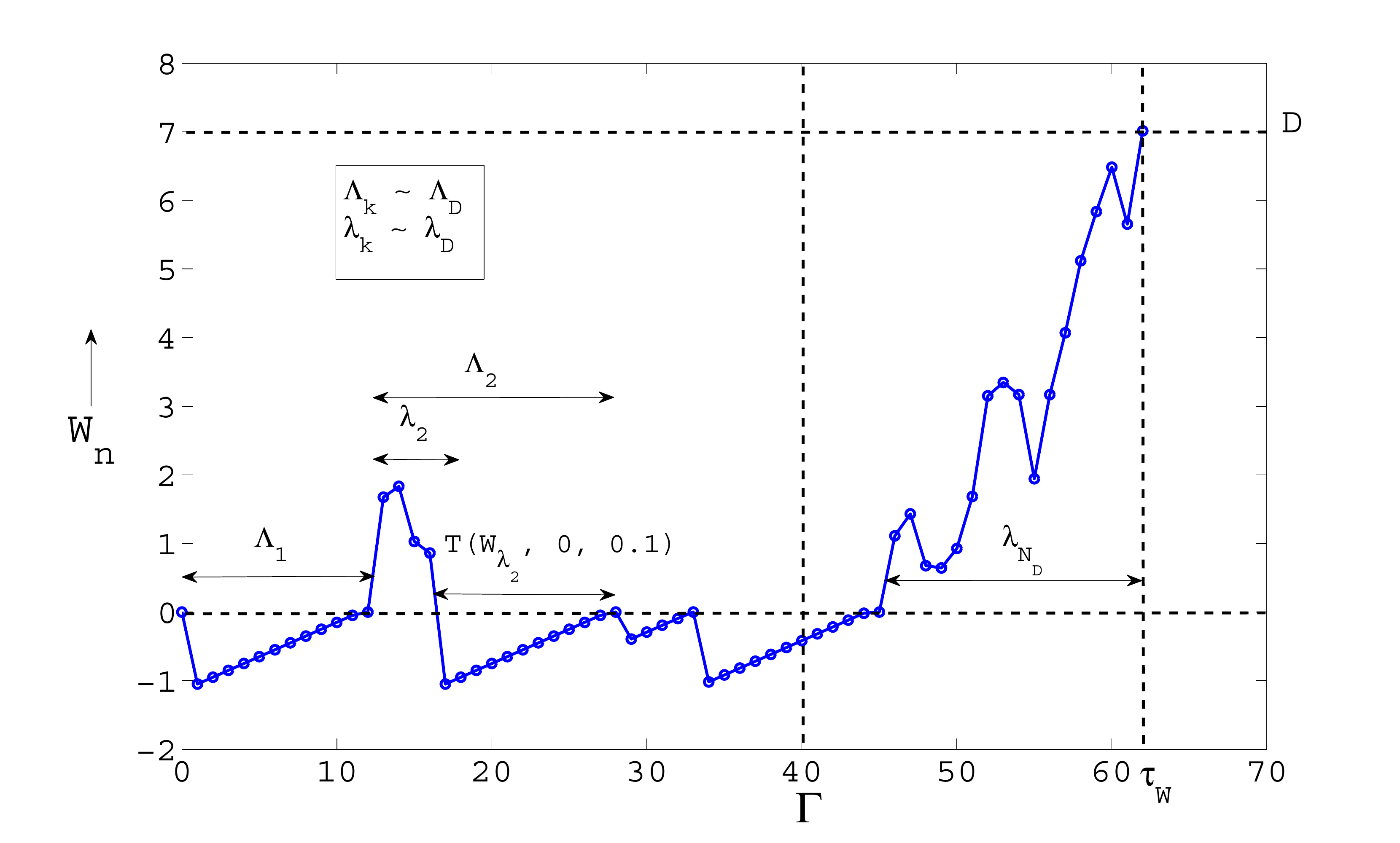}
\caption{Evolution of $W_n$ for $f_0 \sim {\cal N}(0,1)$, $f_1 \sim {\cal N}(0.75,1)$, and $\Gamma=40$, with $D=7$,
$h=\infty$, and $\mu=0.1$. The two-sided tests with distribution of $\Lambdad$ are shown in the figure.
Also shown are the two components of $\Lambdad$: $\lambdad$ and $T(x,y)$.}
\label{fig:DECuSum_evolution_Lambda}
\end{figure}

To summarize, the variables $\lambdad$, $\Lambdad$ and $T(x,y,\mu)$ should be interpreted as follows.
The DE-CuSum algorithm can be seen as a
sequence of two-sided tests, with the stopping time of each two-sided test
distributed accordingly to the law of $\Lambdad$.
Each of the above two-sided tests consists of an SPRT with stopping time distributed accordingly to
the law of $\lambdad$,
and a sojourn of length $T((W_{\lambdad})^{h+},0,\mu)$ corresponding to the time
for which the statistic $W_n$ is below 0, provided
at the stopping time for the SPRT, the accumulated log likelihood is negative, i.e., the event $\{W_{\lambdad} < 0\}$ happens.
See Fig.~\ref{fig:DECuSum_evolution_Lambda}.

The CuSum algorithm
can also be seen as a sequence of SPRTs, with the stopping time of each SPRT
distributed according to the law of $\lambdad$ (see \cite{sieg-seq-anal-book-1985}).

We now provide some results on the mean of $\lambdad$ and $T(x,y,\mu)$
 that will be used in the analysis
of the DE-CuSum algorithm in Sections~\ref{sec:PDCConstraint}, \ref{sec:DECuSumFAR} and \ref{sec:DECuSumCADDWADD}.

If $0< D(f_0 \; \| \; f_1) < \infty$, then
from Corollary 2.4 in \cite{wood-nonlin-ren-th-book-1982},
\begin{equation}\label{eq:EinflambdaInf_finite}
\Expect_\infty[\lambdaInf]< \infty,
\end{equation}
and by Wald's lemma
\begin{equation}\label{eq:EinfWInf_finite}
\Expect_\infty[|W_{\lambdaInf}|] = D(f_0 \; \| \; f_1) \; \Expect_\infty[\lambdaInf] < \infty.
\end{equation}
Also for $h\geq 0$
\begin{equation}\label{eq:EinfWInf_hplus_finite}
\Expect_\infty[|W_{\lambdaInf}^{h+}| ] \leq \Expect_\infty[|W_{\lambdaInf}| ] < \infty,
\end{equation}
where the finiteness follows from \eqref{eq:EinfWInf_finite}.

The lemma below shows that the quantity $\Expect_\infty[\lambdad | W_{\lambdad} < 0]$
is finite for every $D$ and provides a finite upper bound to it that is not a function of the threshold $D$.
This result will be used in the $\PDC$ analysis in Section~\ref{sec:PDCConstraint}.
\smallskip
\begin{lemma}\label{thm:lambdaDUpperBound}
If $0 < D(f_0 \;||\; f_1)< \infty$, then for any $D$, $\Expect_\infty[\lambdad | W_{\lambdad} < 0]$ is
well defined and finite:
\[\Expect_\infty[\lambdad | W_{\lambdad} < 0] \leq \frac{\Expect_\infty[\lambdaInf]}{\Prob_\infty(L(X_1) < 0 )} < \infty.\]
\end{lemma}
\begin{proof}
The proof of the first inequality is provided in the appendix. The second inequality is
true by \eqref{eq:EinflambdaInf_finite} and because $\Prob_\infty(L(X_1)<0)>0$.
\end{proof}
\smallskip

The following lemma provides upper and lower bounds on
$\Expect_\infty[T((W_{\lambdad})^{h+},0,\mu)| W_{\lambdad} < 0]$ that are not a function of the threshold $D$.
The upper bound will be useful in the $\FAR$ analysis in Section~\ref{sec:DECuSumFAR},
and the lower bound will be useful in the $\PDC$ analysis in Section~\ref{sec:PDCConstraint}.
 Define
\begin{equation}\label{eq:TxyLowerBound}
T_L^{(\infty)}(h,\mu) = \frac{\Expect_\infty[|L(X_1)^{h+}| \; \Big| \; L(X_1)<0]}{\mu}  \; \Prob_\infty(L(X_1) < 0),
\end{equation}
 and
\begin{equation}\label{eq:TxyUpperBound}
T_U^{(\infty)}(h,\mu) = \frac{\Expect_\infty[|W_{\lambdaInf}^{h+}|] }{\mu \; \Prob_\infty(L(X_1) < 0) } + 1.
\end{equation}
\smallskip
\begin{lemma}\label{thm:EInfTxy_UpperLower}
If $0< D(f_0 \; \| \; f_1) < \infty$ and $\mu > 0$, then 
\begin{equation}
\begin{split}
T_L^{(\infty)}(h,& \mu)\; \\
& \leq \; \Expect_\infty[T((W_{\lambdad})^{h+},0,\mu) \Big| W_{\lambdad} < 0] \; \\
&\hspace{5cm} \leq \; T_U^{(\infty)}(h,\mu).
\end{split}
\end{equation}
Moreover, $T_U^{(\infty)}(h, \mu) < \infty$, and if $h>0$, then $T_L^{(\infty)}(h, \mu) > 0$.
\end{lemma}
\begin{proof}
The proof is provided in the appendix.
\end{proof}
\smallskip

The next lemma shows that the mean of $\Expect_1[T(W_{\lambdad}^{h+} ,0,\mu) | W_{\lambdad} < 0]$ is finite
under $\Prob_1$ and obtains a finite upper bound to it that is not a function of $D$.
This result will be used for the $\CADD$ and $\WADD$ analysis in Section~\ref{sec:DECuSumCADDWADD}.
Let
\begin{equation}\label{eq:E1TxyUpperBound}
T_U^{(1)}(h,\mu) = \frac{\Expect_\infty[|W_{\lambdaInf}^{h+}|] }{\mu \; \Prob_1(L(X_1) < 0) } + 1.
\end{equation}
\smallskip
\begin{lemma}\label{thm:E1Txy}
If $0< D(f_0 \; \| \; f_1) < \infty$ and $\mu > 0$, then 
\[
\Expect_1[T(W_{\lambdad}^{h+} ,0,\mu) | W_{\lambdad} < 0]
 \; \; \leq \; \; T_U^{(1)}(h,\mu)\; \;  < \; \; \infty.
\]
\end{lemma}
\begin{proof}
The proof is provided in the appendix.
\end{proof}
\smallskip


\subsection{Meeting the $\PDC$ constraint}
\label{sec:PDCConstraint}
In this section we show that the $\PDC$ metric is well defined for the DE-CuSum algorithm.
In general $\PDC(\Psiw)$ will depend on both $D$ and $\mu$ (apart from the obvious dependence on $f_0$ and $f_1$).
But, we show that it is possible to choose a value of $\mu$ that ensures that the $\PDC$ constraint of $\beta$
can be met independent of the choice of $D$.
The latter would be crucial to the asymptotic optimality proof of the DE-CuSum algorithm provided later
in Section~\ref{sec:AsymptoticOpt}.
\smallskip
\begin{theorem}\label{thm:PDCExpr}
\label{thm:PDC}
For fixed values of $D$, $h$, and $\mu>0$, if $0 < D(f_0 \;||\; f_1)< \infty$, then
\begin{equation}\label{eq:PDCExpr}
\begin{split}
\PDC &(\Psiw(D,\mu, h)) = \\
&\frac{\Expect_\infty[\lambdad | W_{\lambdad} < 0]}{\Expect_\infty[\lambdad | W_{\lambdad} < 0] + \Expect_\infty[T((W_{\lambdad})^{h+},0,\mu) \; \Big| \; W_{\lambdad} < 0]}.
\end{split}
\end{equation}
\end{theorem}
\begin{proof}
Consider an alternating renewal process $\{V_n, U_n\}$, i.e, a renewal process with renewal times
$\{V_1, V_1+U_1, V_1+U_1+V_2, \cdots\}$, with $\{V_n\}$ i.i.d. with distribution of $\lambdad$ conditioned
on $\{W_{\lambdad} < 0\}$, and $\{U_n\}$ i.i.d. with distribution of $T((W_{\lambdad})^{h+},0,\mu)$ conditioned
on $\{W_{\lambdad} < 0\}$. Thus,
\[
\Expect_\infty[V_1] = \Expect_\infty[\lambdad | W_{\lambdad} < 0],
\]
and
\[
\Expect_\infty[U_1] = \Expect_\infty[T((W_{\lambdad})^{h+},0,\mu)\; \Big| \; W_{\lambdad} < 0].
\]
Both the means are finite by Lemma~\ref{thm:lambdaDUpperBound} and Lemma~\ref{thm:EInfTxy_UpperLower}.

At time $n$ assign a reward of $R_n=1$ if the renewal cycle in progress has the law of $V_1$, set $R_n=0$ otherwise.
Then by renewal reward theorem,
\[\frac{1}{n}\Expect_\infty\left[\sum_{k=1}^{n-1} R_k \right] \to \frac{\Expect_\infty[V_1]}{\Expect_\infty[V_1] + \Expect_\infty[U_1]}\]
On $\{\tauw \geq n\}$, the total number of observations taken till time $n-1$
has the same distribution as the total reward for the alternating renewal process defined above.
Hence, the expected value of the average reward for both the sequences must have the same limit:
\begin{equation}
\begin{split}
\lim_{n \to \infty}& \frac{1}{n} \Expect_n \left[\sum_{k=1}^{n-1} M_k \Big| \tauw \geq n\right] \\
&= \frac{\Expect_\infty[\lambdad | W_{\lambdad} < 0]}{\Expect_\infty[\lambdad | W_{\lambdad} < 0] + \Expect_\infty[T((W_{\lambdad})^{h+},0,\mu)| W_{\lambdad} < 0]}.
\end{split}
\end{equation}
\end{proof}
\begin{remark}
If $h=0$, then $\Expect_\infty[T((W_{\lambdad})^{h+},0,\mu)| W_{\lambdad} < 0]=0$ and we get the $\PDC$
of the CuSum algorithm that is equal to 1.
\end{remark}
\smallskip

As can be seen from \eqref{eq:PDCExpr}, $\PDC$ is a function of $D$ as well as that of $h$ and $\mu$.
We now show that for any $D$ and $h>0$, the DE-CuSum algorithm can be designed to meet any $\PDC$ constraint of $\beta$.
Moreover, for a given $h>0$, a value of $\mu$ can always be selected such that
the $\PDC$ constraint of $\beta$ is met independent of the choice of $D$.
The latter is convenient not only from a practical point of view, but will also help in the
asymptotic optimality proof of the DE-CuSum algorithm in Section~\ref{sec:AsymptoticOpt}.

\begin{theorem}\label{thm:PDCmuStar}
For the DE-CuSum algorithm, for any choice of $D$ and $h>0$, if $0 < D(f_0 \;||\; f_1)< \infty$, then
we can always choose a value of $\mu$ to meet any given $\PDC$ constraint of $\beta$.
Moreover, for any fixed value of $h>0$,
there exists a value of $\mu$ say $\mu^*(h)$ such that for every $D$,
\[\PDC(\Psiw(D,\mu^*, h)) \leq \beta. \]
In fact any $\mu$ that satisfies
\[\mu \leq \frac{\Expect_\infty[|L(X_1)^{h+}| \; \Big| \; L(X_1)<0] \; \Prob_\infty(L(X_1) < 0)^2}{\Expect_\infty[\lambdaInf]}\frac{\beta}{1-\beta},\]
can be used as $\mu^*$.
\end{theorem}
\begin{proof}
Note that $\Expect_\infty[\lambdad | W_{\lambdad}\leq 0]$ is not affected by the choice of $h$ and $\mu$.
Moreover, from Lemma~\ref{thm:EInfTxy_UpperLower} and \eqref{eq:TxyLowerBound}
\begin{equation*}
\begin{split}
\Expect_\infty &[T((W_{\lambdad})^{h+}, 0,\mu) \; \Big|\;  W_{\lambdad} < 0] \\
& \geq T_L^{(\infty)}(h, \mu) \\
& = \frac{\Expect_\infty[|L(X_1)^{h+}| \; \Big| \; L(X_1)<0]}{\mu}  \; \Prob_\infty(L(X_1) < 0)\\
\end{split}
\end{equation*}
Thus, for a given $D$ and $h$, $\Expect_\infty[T((W_{\lambdad})^{h+},0,\mu)| W_{\lambdad} < 0]$ increases
as $\mu$ decreases. Hence, $\PDC$ decreases as $\mu$ decreases.
Therefore, we can always select a $\mu$ small enough so that the $\PDC$ is smaller than the given
constraint of $\beta$.

Next, our aim is to find a $\mu^*$ such that for every $D$
\begin{eqnarray*}
\frac{\Expect_\infty[\lambdad | W_{\lambdad} < 0 ]}{\Expect_\infty[\lambdad | W_{\lambdad} < 0 ]+
\Expect_\infty [T((W_{\lambdad})^{h+},0,\mu^*) \;  \Big| \; W_{\lambdad} < 0]} \leq \beta,
\end{eqnarray*}
Since, $\PDC$ increases as $\Expect_\infty[\lambdad | W_{\lambdad} < 0]$ increases and
$\Expect_\infty [T((W_{\lambdad})^{h+},0,\mu^*) \;  W_{\lambdad} < 0]$ decreases,
we have from Lemma~\ref{thm:lambdaDUpperBound} and Lemma~\ref{thm:EInfTxy_UpperLower},
\[
\PDC(\Psiw) \leq \frac{\Expect_\infty[\lambdaInf]}{\Expect_\infty[\lambdaInf]+T_L^{(\infty)}(h, \mu) \; \; \Prob_\infty(L(X_1) < 0 )}.
\]

Then, the theorem is proved if we select $\mu$ such that the right hand side of the above equation is less
than $\beta$ or a $\mu$ that satisfies
\[\mu \leq \frac{\Expect_\infty[|L(X_1)^{h+}| \; \Big| \; L(X_1)<0] \; \Prob_\infty(L(X) < 0)^2}{\Expect_\infty[\lambdaInf]}\frac{\beta}{1-\beta}.\]
\end{proof}
\begin{remark}
While the existence of $\mu*$ proved by Theorem~\ref{thm:PDCmuStar} above is critical for asymptotic
optimality of the DE-CuSum algorithm, the estimate it provides when substituted
for $\mu$ in \eqref{eq:PDCExpr} may be a bit conservative. 
In Section~\ref{sec:DECUSUM_Design} we provide a good approximation to $\PDC$ that can be used
to choose the value of $\mu$  in practice. 
In Section~\ref{sec:Trade-off} we provide numerical results 
showing the accuracy of the approximation. 
\end{remark}

\smallskip
\smallskip
\begin{remark}
By Theorem~\ref{thm:PDCmuStar}, for any value of $h$, we can select a value of $\mu$ small enough, so that
any $\PDC$ constraint close to zero can be met with equality. However, meeting the $\PDC$ constraint with
equality may not be possible if $\beta$ is close to 1. This is because
if $h \neq 0$ then
\[\PDC(\Psiw) \leq \frac{\Expect_\infty[\lambdaInf]}{\Expect_\infty[\lambdaInf]+\Prob_\infty(L(X) < 0)} < 1.\]
However, as mentioned earlier, for most practical applications $\beta$ will be close to zero than 1,
and hence this issue will not be encountered. If $\beta$ close to 1 is indeed desired then
the DE-CuSum algorithm can be easily modified to address this issue (by skipping samples only
when the undershoot is larger than a pre-designed threshold).
\end{remark}
\smallskip

\subsection{Analysis of the CuSum algorithm}\label{sec:AnalyCuSum}
In the sections to follow, we will express the performance of the DE-CuSum algorithm in terms
of the performance of the CuSum algorithm. Therefore, in this section we summarize the performance
of the CuSum algorithm.

It is well known (see \cite{lord-amstat-1971}, \cite{sieg-seq-anal-book-1985}, \cite{veer-bane-elsevierbook-2013}),
that
\begin{equation}\label{eq:CuSumDelayCADDWADD}
\CADD(\Psic) = \WADD(\Psic) = \Expect_1[\tauc-1].
\end{equation}
From \cite{lord-amstat-1971}, if $0 < D(f_1 \; ||\; f_0)<\infty$, then $\Expect_1[\tauc]<\infty$. Moreover, if
$\{\lambda_1, \lambda_2, \cdots\}$ are i.i.d. random variables each with distribution of $\lambdad$, then
by Wald's lemma \cite{sieg-seq-anal-book-1985}
\begin{equation}\label{eq:CuSumDelayE1}
\Expect_1[\tauc] = \Expect_1\left[\sum_{k=1}^{N} \lambda_k\right] =  \Expect_1[N] \;\Expect_1[\lambdad],
\end{equation}
where $N$ is the number of two-sided tests (SPRTs)--each with distribution of $\lambdad$--executed before the change is declared.

It is also shown in \cite{lord-amstat-1971} that $0 < D(f_1 \; ||\; f_0)<\infty$ is also sufficient to guarantee
$\Expect_\infty[\tauc]< \infty$ and $\FAR(\Psic)>0$.  Moreover,
\begin{equation}\label{eq:CuSumFAR}
\Expect_\infty[\tauc] = \Expect_\infty\left[\sum_{k=1}^{N} \lambda_k\right] =  \Expect_\infty[N] \;\Expect_\infty[\lambdad].
\end{equation}

The proof of the following theorem can be found in \cite{lord-amstat-1971} and \cite{lai-ieeetit-1998}.
\begin{theorem}\label{thm:CuSumOptLordenLai}
If $0 < D(f_1 \; ||\; f_0)<\infty$, then with $D=\log\frac{1}{\alpha}$,
\[
\FAR(\Psic) \leq \alpha,
\]
and as $\alpha \to 0$,
\[
\CADD(\Psic) = \WADD(\Psic) = \Expect_1[\tauc-1] \sim \frac{|\log \alpha|}{D(f_1 \; ||\; f_0)}.
\]
Thus, the CuSum algorithm is asymptotically optimal for both Problem~\ref{prob:DELorden} and Problem~\ref{prob:DEPollak}
because for any stopping time $\tau$ with $\FAR(\tau) \leq \alpha$,
\begin{equation}\label{eq:lordenLowerBound}
\WADD(\tau) \geq \CADD(\tau) \geq \frac{|\log \alpha|}{D(f_1 \; ||\; f_0)} \Big( 1+o(1) \Big),
\end{equation}
as $\alpha \to 0$.
\end{theorem}

\subsection{$\FAR$ for the DE-CuSum algorithm}
\label{sec:DECuSumFAR}
In this section we characterize the false alarm rate of the DE-CuSum algorithm.
The following theorem shows that for a fixed $D$, $\mu$ and $h$,
if the DE-CuSum algorithm and the CuSum algorithm are applied to the same
sequence of random variables, then sample-pathwise,
the DE-CuSum statistic $W_n$ is always below the CuSum statistic $C_n$.
Thus, the DE-CuSum algorithm crosses
the threshold $D$ only after the CuSum algorithm has crossed it.

\smallskip
\begin{lemma}\label{lem:SamplePathwise}
Under any $\Prob_n, \;n \geq 1$ and under $\Prob_\infty$,
\[C_n \geq W_n.\]
Thus
\[\tauc \leq \tauw.\]
\end{lemma}
\begin{proof}
This follows directly from the definition of the DE-CuSum algorithm. If a sequence
of samples causes the statistic of the DE-CuSum algorithm to go above $D$, then since
all the samples are utilized in the CuSum algorithm, the same sequence must also cause the
CuSum statistic to go above $D$.
\end{proof}

It follows as a corollary of Lemma~\ref{lem:SamplePathwise} that
\[\Expect_\infty[\tauc] \leq \Expect_\infty[\tauw].\]
The following theorem shows that these quantities are finite and
also provides an estimate for $\FAR(\Psiw)$.
\smallskip
\begin{theorem}\label{thm:FARMean}
For any fixed $h$ (including $h=\infty$) and $\mu>0$, if
\[
0< D(f_0 \;||\; f_1) < \infty \ \ \ \mbox{ and } \ \ \ 0< D(f_1\; ||\; f_0) < \infty,
\]
then with $D=\log\frac{1}{\alpha}$,
\[
\FAR(\Psiw) \leq \FAR(\Psic) \leq \alpha.
\]
Moreover, for any $D$
\begin{equation}\label{eq:Etauw_intermsof_Etauc}
\begin{split}
\Expect_\infty[\tauw] &= \frac{\Expect_\infty[\Lambdad]}{\Prob_\infty(W_{\lambdad}>0)} \\
&=\frac{\Expect_\infty[\lambdad]}{\Prob_\infty(W_{\lambdad}>0)} +
\frac{\Expect_\infty[T((W_{\lambdad})^{h+},0,\mu) \; \indic_{\{W_{\lambdad} < 0\}}]}{\Prob_\infty(W_{\lambdad}>0)}\\
&=\Expect_\infty[\tauc] +
\frac{\Expect_\infty[T((W_{\lambdad})^{h+},0,\mu) \; \indic_{\{W_{\lambdad} < 0\}}]}{\Prob_\infty(W_{\lambdad}>0)}
\end{split}
\end{equation}
and as $D\to \infty$,
\begin{equation}\label{eq:ratioOfFARs}
\frac{\FAR(\Psiw)}{\FAR(\Psic)} \to \frac{\Expect_\infty[\lambdaInf]}{\Expect_\infty[\lambdaInf] +
\Expect_\infty[T((W_{\lambdaInf})^{h+},0,\mu)]},
\end{equation}
where $\lambdaInf$ is the variable $\lambdad$ with $D=\infty$.
The limit in \eqref{eq:ratioOfFARs} is strictly less than 1 if $h>0$.
\end{theorem}
\begin{proof}
For a fixed $D$, let $\Nd$ be the number of two-sided tests of distribution $\Lambdad$
executed before the change is declared in the DE-CuSum algorithm.
Then, if $\{\Lambda_1, \Lambda_2, \cdots\}$ is a sequence of random variables each with distribution of $\Lambdad$, then
\[
\Expect_\infty[\tauw] = \Expect_\infty\left[\sum_{k=1}^{\Nd} \Lambda_k\right]
\]

Because of the renewal nature of the DE-CuSum algorithm,
\[
\Expect_\infty[\Nd] = \Expect_\infty[N],
\]
where $N$ is the number of SPRTs used in the CuSum algorithm.
Thus from \eqref{eq:CuSumFAR},
\[
\Expect_\infty[\Nd] = \Expect_\infty[N] \leq \Expect_\infty[\tauc] < \infty.
\]

Further from \eqref{eq:Lambda},
\begin{equation}\label{eq:Lambda_Intermsof_lambda}
\Expect_\infty[\Lambdad] = \Expect_\infty[\lambdad] + \Expect_\infty[T((W_{\lambdad})^{h+},0,\mu) \; \indic_{\{W_{\lambdad} < 0\}}].
\end{equation}
From \eqref{eq:CuSumFAR} again
\[
\Expect_\infty[\lambdad] \leq \Expect_\infty[\tauc] < \infty.
\]
Moreover from Lemma~\ref{thm:EInfTxy_UpperLower}
\begin{equation*}
\begin{split}
\Expect_\infty &[T((W_{\lambdad})^{h+},0,\mu) \; \indic_{\{W_{\lambdad} < 0\}}] \\
&\leq \Expect_\infty [T((W_{\lambdad})^{h+},0,\mu) \; | \; W_{\lambdad} < 0] \\
&\leq T_U^{(\infty)}(h,\mu) < \infty.\\
\end{split}
\end{equation*}
Thus, $\Expect_\infty[\Lambdad]<\infty$ and
\[
\Expect_\infty[\tauw] = \Expect_\infty\left[\sum_{k=1}^{\Nd} \Lambda_k\right] =  \Expect_\infty[\Nd] \;\Expect_\infty[\Lambdad] < \infty.
\]
It follows as a corollary of Lemma~\ref{lem:SamplePathwise} and Theorem \ref{thm:CuSumOptLordenLai}
that for $D=\log\frac{1}{\alpha}$,
\[\FAR(\Psiw) \leq \FAR(\Psic) \leq \alpha.\]

Since, $\Nd$ is $\text{Geom}(\Prob_\infty(W_{\lambdad}>0))$, \eqref{eq:Etauw_intermsof_Etauc} follows from \eqref{eq:Lambda_Intermsof_lambda} and \eqref{eq:CuSumFAR}.

Further, since $\Expect_\infty[\Nd] = \Expect_\infty[N]$, we have
\[
\frac{\Expect_\infty[\tauc]}{\Expect_\infty[\tauw]} = \frac{\Expect_\infty[N] \;\Expect_\infty[\lambdad]}{\Expect_\infty[\Nd] \;\Expect_\infty[\Lambdad]} = \frac{\Expect_\infty[\lambdad]}{\Expect_\infty[\Lambdad]}.
\]
If
\[\mathcal{C} = \{W_n \mbox{ reaches below zero only after touching } D\},\]
then as $D \to \infty$, $\Prob_\infty(\mathcal{C}) \to 0$ and since
$T((W_{\lambdaInf})^{h+},0,\mu)$ and $\lambdaInf$ are integrable under $\Prob_\infty$,
\[\Expect_\infty [T((W_{\lambdaInf})^{h+},0,\mu) \; ; \; \mathcal{C}] \to 0,\]
and
\[\Expect_\infty [\lambdaInf \; ; \; \mathcal{C}] \to 0.\]
Thus, as $D\to \infty$,
\[
\frac{\Expect_\infty[\tauc]}{\Expect_\infty[\tauw]} \to \frac{\Expect_\infty[\lambdaInf]}{\Expect_\infty[\LambdaInf]}
= \frac{\Expect_\infty[\lambdaInf]}{\Expect_\infty[\lambdaInf] + \Expect_\infty [T((W_{\lambdaInf})^{h+},0,\mu) ]}.
\]
The limit is clearly less than 1 if $h>0$.
\end{proof}
\smallskip
\begin{remark}
Thus, unlike the Bayesian setting where the $\PFA$ of the DE-Shiryaev algorithm converges to the $\PFA$ of the
Shiryaev algorithm, here, the $\FAR$ of the DE-CuSum algorithm is strictly less than the $\FAR$ of the
CuSum algorithm. Moreover, for large $D$, the right side of \eqref{eq:ratioOfFARs} is approximately
the $\PDC$ achieved. Thus, \eqref{eq:ratioOfFARs} shows that, asymptotically as $D\to \infty$,
the ratio of the $\FAR$s is approximately equal to the $\PDC$. This also shows that
one can set the threshold in the DE-CuSum algorithm to a value much smaller than
$D=\frac{1}{\alpha}$ to meet the $\FAR$ constraint with equality, and as a result
get a better delay performance. This latter fact will be used in
obtaining the numerical results in Section~\ref{sec:Trade-off}.
\end{remark}

\subsection{$\CADD$ and $\WADD$ of the DE-CuSum algorithm}
\label{sec:DECuSumCADDWADD}
We now provide expressions for $\CADD$ and $\WADD$ of the DE-CuSum algorithm
The main content of Theorem~\ref{thm:CADD} and Theorem~\ref{thm:WADD} below is,
that for each value of $D$, the $\CADD$ and $\WADD$ of the DE-CuSum algorithm is
within a constant of the corresponding performance of the CuSum algorithm. This constant
is independent of the choice of $D$,
and as a result the delay performances of the two algorithms are asymptotically the same.

The results depend on the following fundamental lemma. The lemma says that when the change
happens at $n=1$, then the average delay of the DE-CuSum algorithm starting with $W_0 = x > 0$,
is upper bounded by the average delay of the algorithm when $W_0=0$, plus a constant.
Let
\[\tauw(x) = \inf\{n\geq 1: W_n > D; W_0=x\}.\]
Here, $W_n$ is the DE-CuSum statistic and evolves according the description of the algorithm in
Section \ref{sec:DECUSUMALGO}. Thus, $\tauw(x)$ is the first time for the DE-CuSum algorithm
to cross $D$, when starting at $W_0=x$. Clearly, $\tauw(x) = \tauw$ if $x=0$.
\begin{lemma}\label{thm:resettingLemma}
Let $0< D(f_1\; ||\; f_0) < \infty$ and $0 \leq x < D$. Then,
\[\Expect_1[\tauw(x)] \leq \Expect_1[\tauw] + T_U^{(1)}(h,\mu),\]
where, $T_U^{(1)}(h,\mu)$ is an upper bound to the variable $T(x,y)$ (see \eqref{eq:E1TxyUpperBound}).
Moreover if $h<\infty$, then
\[\Expect_1[\tauw(x)] \leq \Expect_1[\tauw] + \lceil h/\mu \rceil.\]
\end{lemma}
\begin{proof}
The proof is provided in the appendix.
\end{proof}

We first provide the result for the $\CADD$s of the two algorithms.
\begin{theorem}\label{thm:CADD}
Let
\[
0< D(f_0 \;||\; f_1) < \infty \ \ \ \mbox{ and } \ \ \ 0< D(f_1\; ||\; f_0) < \infty.
\]
Then, for fixed values of $\mu>0$ and $h$, and for each $D$,
\[\CADD(\Psiw) \leq \CADD(\Psic) + K_1 ,\]
where $K_1$ is a constant not a function of $D$.
Thus as $D\to \infty$,
\[\CADD(\Psiw) \leq \CADD(\Psic) + O(1) .\]
\end{theorem}
\begin{proof}
If the change happens at $n=1$ then
\begin{equation*}
\begin{split}
\Expect_1 &[\tauw-1 | \tauw\geq 1] = \Expect_1[\tauw] -1 \leq \Expect_1[\tauw].
\end{split}
\end{equation*}

Let the change happen at time $n>1$. Then on $\{W_{n-1} \geq 0\}$, by Lemma~\ref{thm:resettingLemma},
the average delay is
bounded from above by $\Expect_1[\tauw] + T_U^{(1)}(h,\mu)$, and on $\{W_{n-1} < 0\}$ the average delay
is bounded from above by $\Expect_1[\tauw]$ plus the maximum possible average time
spent by the DE-CuSum statistic below $0$ under $\Prob_\infty$, which is $T_U^{(\infty)}(h,\mu)$.
Thus, from Lemma~\ref{thm:EInfTxy_UpperLower}, for $n>1$,
\begin{equation*}
\begin{split}
\Expect_n &[\tauw-n | \tauw\geq n] \\
&\leq \left(\Expect_1[\tauw]+T_U^{(1)}(h,\mu)\right)\; \Prob_\infty(W_{n-1}\geq 0) \\
&+\left( \Expect_1[\tauw] + T_U^{(\infty)}(h,\mu) \right) \Prob_\infty(W_{n-1}< 0)
\end{split}
\end{equation*}
Thus, for all $n \geq 1$
\begin{equation*}
\Expect_n [\tauw-n | \tauw\geq n] \leq \Expect_1[\tauw] + T_U^{(1)}(h,\mu) + T_U^{(\infty)}(h,\mu).
\end{equation*}
Since, the right hand side of the above equation is not a function of $n$ we have
\begin{equation*}
\CADD(\Psiw) \leq \Expect_1[\tauw] + T_U^{(1)}(h,\mu) + T_U^{(\infty)}(h,\mu).
\end{equation*}
Following Theorem~\ref{thm:FARMean} and its proof, it is easy to see that
\begin{equation*}
\begin{split}
\Expect_1[\tauw] = \Expect_1[\tauc] + \frac{\Expect_1[T((W_{\lambdad})^{h+},0,\mu) \; \indic_{\{W_{\lambdad} < 0\}}]}{\Prob_1(W_{\lambdad}>0)}.
\end{split}
\end{equation*}
From Lemma~\ref{thm:E1Txy} and the fact that $\Prob_1(W_{\lambdad}>0) > \Prob_1(W_{\lambdaInf}>0)$ we have
\[
\frac{\Expect_1[T((W_{\lambdad})^{h+},0,\mu) \; \indic_{\{W_{\lambdad} < 0\}}]}{\Prob_1(W_{\lambdad}>0)}
\leq \frac{T_U^{(1)}(h,\mu)}{\Prob_1(W_{\lambdaInf}>0)}.
\]
Also from \eqref{eq:CuSumDelayCADDWADD} we have $\CADD(\Psic) = \Expect_1[\tauc-1]$. Thus,
\begin{equation*}
\begin{split}
\CADD(\Psiw) &\leq \; \CADD(\Psic) \\
&+ \frac{T_U^{(1)}(h,\mu)}{\Prob_1(W_{\lambdaInf}>0)} +
T_U^{(1)}(h,\mu) + T_U^{(\infty)}(h,\mu) + 1 .
\end{split}
\end{equation*}
This proves the theorem.
\end{proof}
\smallskip
\begin{remark}
Note that the above theorem is valid even if $h$ is not finite. In contrast,
as we will see below, the $\WADD(\Psiw) = \infty$ if $h = \infty$. As a result,
we need a bound on the number of samples skipped for finiteness of worst case delay
according to the criterion of Lorden.
\end{remark}

We now express the $\WADD$ of the DE-CuSum algorithm in terms of the $\WADD$ of the CuSum algorithm.

\begin{theorem}\label{thm:WADD}
Let
\[
0< D(f_0 \;||\; f_1) < \infty \ \ \ \mbox{ and } \ \ \ 0< D(f_1\; ||\; f_0) < \infty.
\]
Then, for fixed values of $\mu>0$ and $h<\infty$, and for each $D$,
\[\WADD(\Psiw) \leq \WADD(\Psic) + K_2 ,\]
where $K_2$ is a constant not a function of $D$.
Thus, as $D\to \infty$,
\[\WADD(\Psiw) \leq \WADD(\Psic) + O(1) .\]
\end{theorem}
\begin{proof}
From Lemma~\ref{thm:resettingLemma}, it follows that for $n>1$
\[\text{ess sup}\; \Expect_n \left[ (\tauw-n)^+ | I_{n-1} \right] = \lceil h/\mu \rceil + \Expect_1[\tauw].\]
Since the right hand side is not a function of $n$ and it is greater than $\Expect_1 [\tauw-1]$, we have
\[\WADD(\Psiw) = \lceil h/\mu \rceil + \Expect_1[\tauw].\]

Thus, from the proof of theorem above and \eqref{eq:CuSumDelayCADDWADD}
\begin{equation*}
\begin{split}
\Expect_1[\tauw] &\leq \Expect_1[\tauc] + \frac{T_U^{(1)}(h,\mu)}{\Prob_1(W_{\lambdaInf}>0)} \\
&= \WADD(\Psic) + \frac{T_U^{(1)}(h,\mu)}{\Prob_1(W_{\lambdaInf}>0)}  + 1,
\end{split}
\end{equation*}
and we have
\begin{equation*}
\begin{split}
\WADD(\Psiw) \leq \WADD(\Psic) + \frac{T_U^{(1)}(h,\mu)}{\Prob_1(W_{\lambdaInf}>0)}
+ \frac{h}{\mu} + 2.
\end{split}
\end{equation*}
This proves the theorem.
\end{proof}

\smallskip
The following corollary follows easily from Theorem~\ref{thm:CuSumOptLordenLai}, Theorem~\ref{thm:CADD} and Theorem~\ref{thm:WADD}.

\begin{corrly}\label{thm:CADDWADDCorollary}
If $0 < D(f_1 \; ||\; f_0)<\infty$ and $0 < D(f_0 \; ||\; f_1)<\infty$, then
for fixed values of $\mu$ and $h$, including the case of $h=\infty$ (no truncation), as $D\to \infty$,
\[\CADD(\Psiw) \sim \frac{D}{D(f_1 \; ||\; f_0)}.\]
Moreover, if $h<\infty$, then as $D\to \infty$,
\[\WADD(\Psiw) \sim \frac{D}{D(f_1 \; ||\; f_0)} .\]
\end{corrly}


%
\subsection{Asymptotic optimality of the DE-CuSum algorithm}
\label{sec:AsymptoticOpt}
We now use the results from the previous sections to show that the DE-CuSum algorithm
is asymptotically optimal.

The following theorem says that for a given $\PDC$ constraint of $\beta$,
the DE-CuSum algorithm is asymptotically optimal
for both Problem~\ref{prob:DELorden} and Problem~\ref{prob:DEPollak},
as $\alpha \to 0$, for the following reasons:
\begin{itemize}
\item the $\PDC$ of the DE-CuSum algorithm can be designed to meet the constraint independent of the choice of $D$,
\item the $\CADD$ and $\WADD$ of the DE-CuSum algorithm approaches the corresponding performances
of the CuSum algorithm, 
\item the $\FAR$ of the DE-CuSum algorithm is always better than that
of the CuSum algorithm, and
\item the CuSum algorithm is asymptotically optimal for both Problem~\ref{prob:DELorden}
and Problem~\ref{prob:DEPollak}, as $\alpha \to 0$.
\end{itemize}
\smallskip
\begin{theorem}\label{thm:AsymptOpt}
Let $0 < D(f_1 \; ||\; f_0)<\infty$ and $0 < D(f_0 \; ||\; f_1)<\infty$.
For a given $\alpha$, set $D=\log\frac{1}{\alpha}$, then for any choice of $h$ and $\mu$,
\[\FAR(\Psiw) \leq \FAR(\Psic) \leq \alpha.\]
For a given $\beta$, and for any given $h$, it is possible to select $\mu=\mu^*(h)$ such
that $\forall D$, and hence even with $D=\log\frac{1}{\alpha}$,
\[\PDC(\Psiw) \leq \beta.\]
Moreover, for each fixed $\beta$, for any $h$ and with $\mu^*(h)$ selected to meet this $\PDC$
constraint of $\beta$, as $\alpha \to 0$ (or $D\to \infty$ because $D=\log\frac{1}{\alpha}$),
\[\CADD(\Psiw(\log\frac{1}{\alpha}, h, \mu^*(h))) \sim \CADD(\Psic) \sim \frac{|\log \alpha|}{D(f_1\;||\;f_0)}.\]
Furthermore, if the $h$ chosen above is finite, then
\[\WADD(\Psiw(\log\frac{1}{\alpha}, h, \mu^*(h))) \sim \WADD(\Psic) \sim \frac{|\log \alpha|}{D(f_1\;||\;f_0)}.\]
\end{theorem}
\begin{proof}
The result on $\FAR$ follows from Theorem~\ref{thm:FARMean}. The fact that one can select a $\mu=\mu^*(h)$
to meet the $\PDC$ constraint independent of the choice of $D$ follows from Theorem~\ref{thm:PDCmuStar}.
Finally, the delay asymptotics follow from Theorem~\ref{thm:CADD}, Theorem~\ref{thm:WADD} and Corollary~\ref{thm:CADDWADDCorollary}.
\end{proof}
Since, by Theorem~\ref{thm:CuSumOptLordenLai},
$\frac{|\log \alpha|}{D(f_1\;||\;f_0)}$ is the best possible asymptotics performance for any
given $\FAR$ constraint of $\alpha$, the above statement establishes the asymptotic optimality of the
DE-CuSum algorithm for both Problem~\ref{prob:DEPollak} and Problem~\ref{prob:DELorden}.

\subsection{Design of the DE-CuSum algorithm}
\label{sec:DECUSUM_Design}
We now discuss how to set the parameters $\mu$, $h$ and $D$ so as to meet a given $\FAR$
constraint of $\alpha$ and a $\PDC$ constraint of $\beta$. 

Theorem~\ref{thm:FARMean} provides the guideline for choosing $D$: for any $h, \mu$,
\[\mbox{ if } \ \ D=\log\frac{1}{\alpha}   \ \ \ \ \ \mbox{   then   } \ \ \FAR(\Psiw) \leq \alpha.\]

As discussed earlier, Theorem~\ref{thm:PDCmuStar}
provides a conservative estimate of the $\PDC$. For practical purposes,
we suggest using
the following approximation for $\PDC$:
\begin{equation}\label{eq:PDCApprox1}
\PDC \approx \frac{\Expect_\infty[\lambdaInf]}{\Expect_\infty[\lambdaInf] +
\Expect_\infty[\lceil \frac{W_{\lambdaInf}^{h+}}{\mu}\rceil]}. 
\end{equation}
For large values of $D$, \eqref{eq:PDCApprox1} will indeed provide a good estimate
of the $\PDC$. We note that $\Expect_\infty[\lambdaInf]$ 
can be computed numerically; see Corollary 2.4 in \cite{wood-nonlin-ren-th-book-1982}.
 
If $h=\infty$, then using \eqref{eq:EinfWInf_finite} we can further simplify \eqref{eq:PDCApprox1} to 
\begin{equation}\label{eq:PDCApprox2}
\PDC \approx \frac{\Expect_\infty[\lambdaInf]}{\Expect_\infty[\lambdaInf] +
\frac{\Expect_\infty[|W_{\lambdaInf}|]}{\mu}} = \frac{\mu}{\mu + D(f_0 \;|| \;f_1)}.
\end{equation}
Thus, to ensure $\PDC \leq \beta$, the approximation above suggests 
selecting $\mu$ such that 
\[\mu \leq \frac{\beta}{1-\beta} D(f_0 \;|| \;f_1).\]
In Section~\ref{sec:Trade-off} we provide numerical results that shows that the approximation 
\eqref{eq:PDCApprox2} indeed provides a good estimate of the $\PDC$ when $h=\infty$.   

\section{Trade-off curves}
\label{sec:Trade-off}
The asymptotic optimality of the DE-CuSum algorithm for all $\beta$
does not guarantee good performance for moderate values of FAR. In
Fig.~\ref{fig:DECuSum_Tradeoff}, we plot the trade-off curves
for the CuSum algorithm and the DE-CuSum algorithm, obtained using simulations.
We plot the performance of the DE-CuSum algorithm for two different
PDC constraints: $\beta=0.5$ and $\beta=0.25$.
For simplicity we restrict ourself to the $\CADD$ performance for $h=\infty$ in this section.
Similar performance comparisons can be obtained for $\CADD$ with $h<\infty$, and for $\WADD$.

Each of the curves for the DE-CuSum algorithm in Fig.~\ref{fig:DECuSum_Tradeoff}
is obtained in the following way.
Five different threshold values for $D$ were arbitrarily selected. For each threshold value,
a large value of $\gamma$ was chosen, and
the DE-CuSum algorithm was simulated and the fraction of time the observations are taken
before change was computed. Specifically, $\gamma$ was increased in the multiples of $100$
and an estimate of the $\PDC$ was obtained by Monte Carlo Simulations. The value of $\mu$
was so chosen that the $\PDC$ value obtained in simulations was slightly below the constraint
$\beta=0.5$ or $0.25$.
For this value of $\mu$ and for the chosen threshold, the $\FAR$ was computed by selecting
the change time to be $\gamma=\infty$ (generating random numbers from
$f_0 \sim {\cal N}(0,1)$). The $\CADD$ was then computed for the above choice of $\mu$ and
$D$ by varying the value of $\gamma$ from $1,2,\ldots$ and recording the maximum of
the conditional delay. The maximum was achieved in the first five slots.

As can be seen from the figure, a PDC of $0.5$
(using only 50\% of the samples in the long run) can be achieved using
the DE-CuSum algorithm with a small penalty
on the delay. If we wish to achieve a PDC of $0.25$,
then we have to incur a significant penalty (of approximately 6 slots in Fig.~\ref{fig:DECuSum_Tradeoff}).
But, note that the difference of delay with the CuSum algorithm remains fixed as $\mathrm{FAR} \to 0$.
This is due to the result reported in Theorem~\ref{thm:CADD} and this is precisely the reason
the DE-CuSum algorithm is asymptotic optimal.
The trade-off between $\CADD$ and $\FAR$ is a function of the K-L divergence between
the pdf's $f_1$ and $f_0$: the larger the K-L divergence the more is the fraction of samples
that can dropped for a given loss in delay performance.

In Fig.~\ref{fig:CuSum_DECuSum_Comp} we compare the performance of the DE-CuSum algorithm
with the \textit{fraction sampling} scheme, in which, to
achieve a PDC of $\beta$, the CuSum algorithm is employed,
and a sample is chosen with probability $\beta$ for decision making.
Note that this scheme skips samples without exploiting any knowledge about the state of the system.
As seen in Fig. \ref{fig:CuSum_DECuSum_Comp},
the DE-CuSum algorithm performs considerably better than the fractional sampling scheme.
Thus, the trade-off curves show that the DE-CuSum algorithm has good performance
even for moderate FAR, when the PDC constraint is moderate.

\begin{figure}[htb]
\center
\includegraphics[width=8cm, height=5cm]{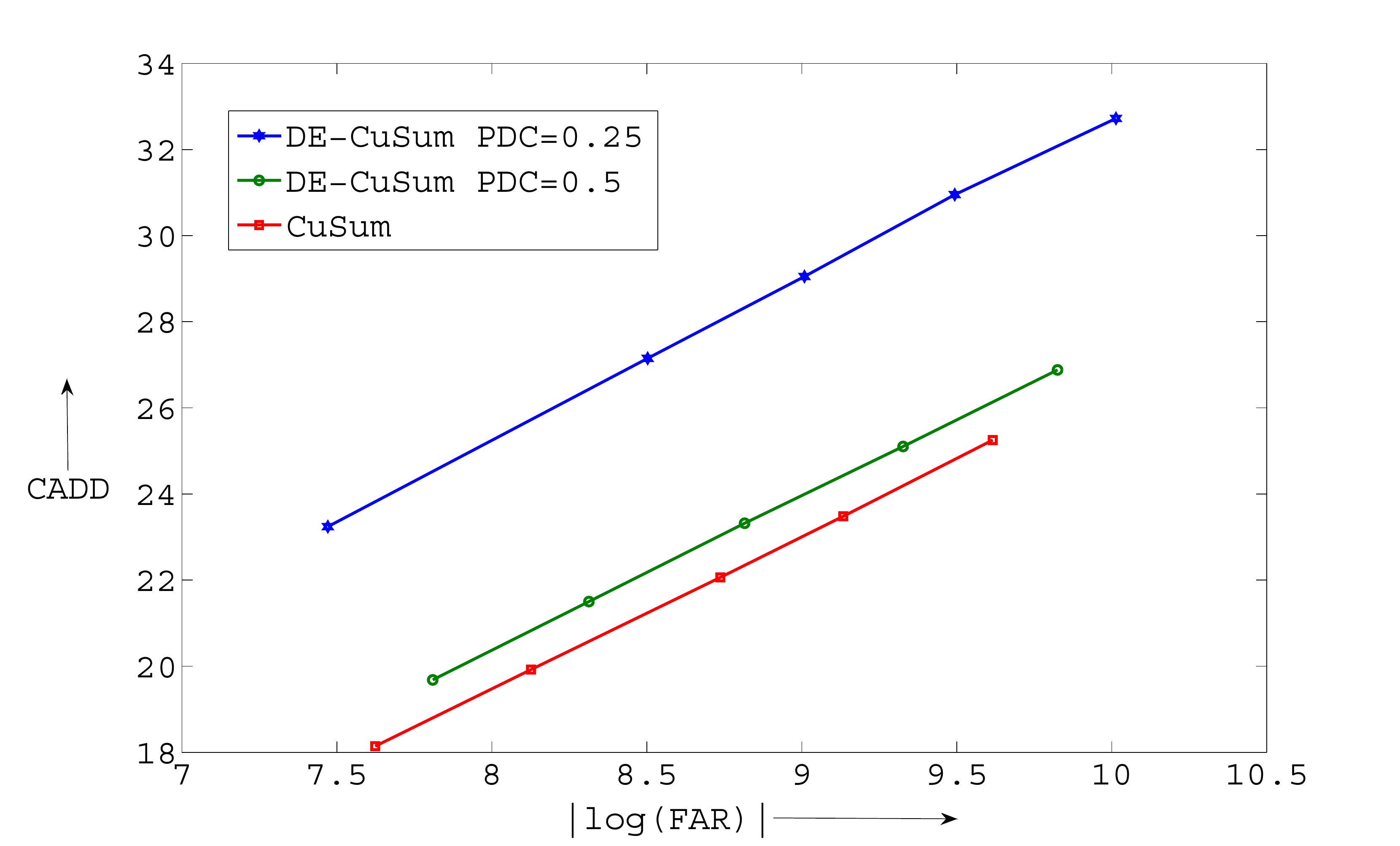}
\caption{Trade-off curves for the DE-CuSum algorithm for $\mathrm{PDC}=0.25, 0.5$, with $f_0 \sim {\cal N}(0,1)$ and $f_1 \sim {\cal N}(0.75,1)$.}
\label{fig:DECuSum_Tradeoff}
\end{figure}

\begin{figure}[htb]
\center
\includegraphics[width=8cm, height=5cm]{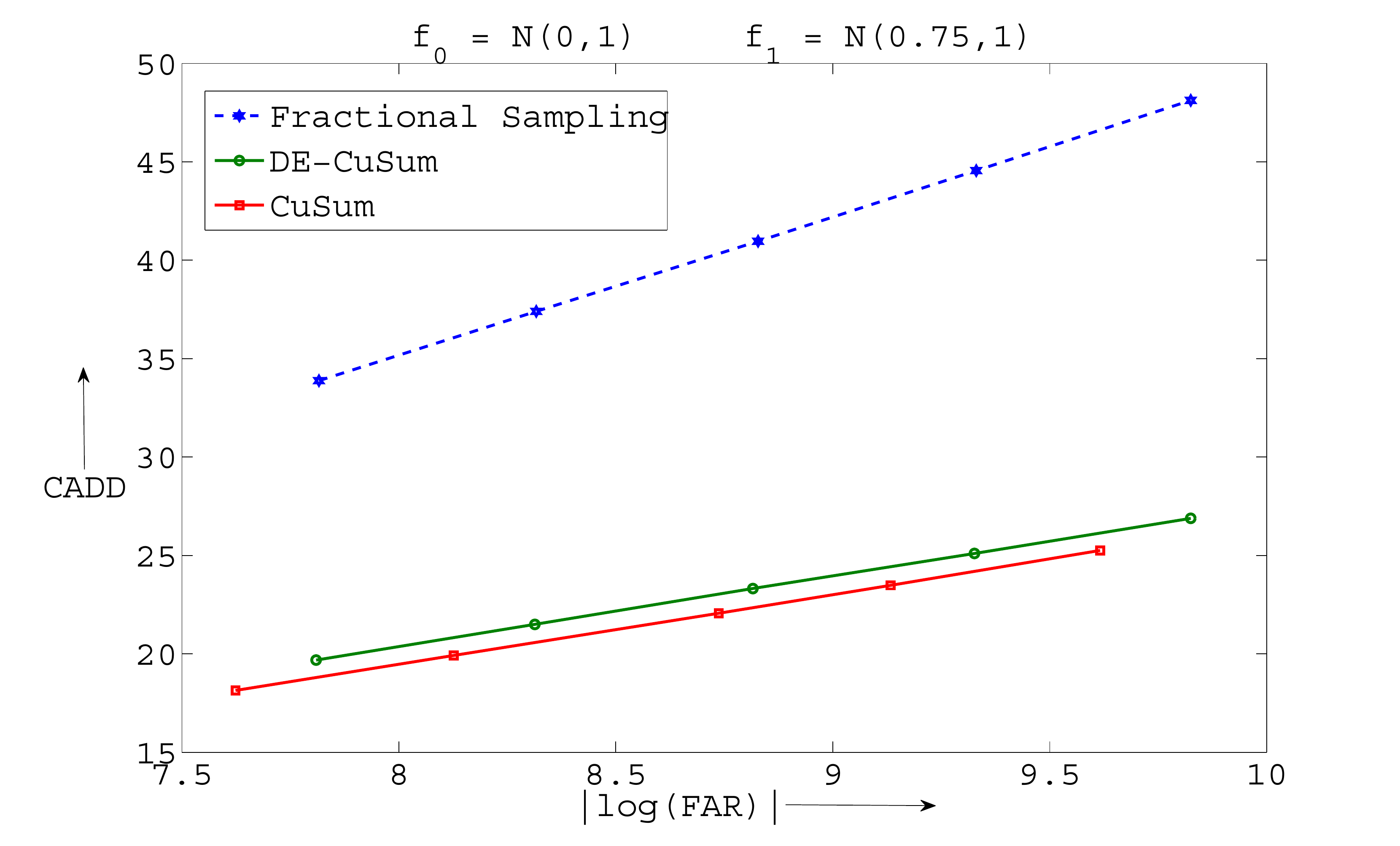}
\caption{Comparative performance of the DE-CuSum algorithm with the CuSum algorithm and the fractional-sampling scheme: $\mathrm{PDC}=0.5$, with $f_0 \sim {\cal N}(0,1)$ and $f_1 \sim {\cal N}(0.75,1)$.}
\label{fig:CuSum_DECuSum_Comp}
\end{figure}

We now provide numerical results that shows that \eqref{eq:PDCApprox2} provides a good estimate 
for the $\PDC$. We use the following parameters: $f_0 \sim {\cal N}(0,1)$, $f_1 \sim {\cal N}(0.75,1)$
and set $h=\infty$. 
In Table \ref{tab:PDCapproxFixedmu}, we fix the value of $\mu$ and vary $D$ and compare the 
$\PDC$ obtained using simulations and the one obtained using \eqref{eq:PDCApprox2}, that is 
using the approximation $\PDC \approx \frac{\mu}{\mu+D(f_0||f_1)}$. 
We see that the approximation becomes more accurate as $D$ increases. We also note that 
the $\PDC$ obtained using simulations does not converge to $\frac{\mu}{\mu+D(f_0||f_1)}$, 
even as $D$ becomes large,   
because of the effect of the presence of a ceiling function in the $\PDC$ expression; 
see \eqref{eq:TXYCeil} and \eqref{eq:PDCExpr}. 

In Table~\ref{tab:PDCapproxFixedD}, we next fix a large value of $D$, 
specifically $D=6$, for which the $\PDC$ approximation is most 
accurate in Table~\ref{tab:PDCapproxFixedmu}, and check the accuracy of the approximation 
$\frac{\mu}{\mu+D(f_0||f_1)}$ by varying $\mu$. We see in the table
that the approximation is more accurate for small values of $\mu$. This is due to the fact 
that the effect of the ceiling function in the $\PDC$ \eqref{eq:TXYCeil}, \eqref{eq:PDCExpr}
is negligible when $\mu$ is small.

\begin{table}[ht]
  \centering
\subfloat[Fixed $\mu$]
{ \label{tab:PDCapproxFixedmu}
  \begin{tabular}{|l|l|l|l|l|l|l|l|l|l|l|l|l|}
\hline
&&\multicolumn{2}{c|}{$\PDC$}\\
\cline{1-4}
$D$ & $\mu$ & Simulations &   Approx      \eqref{eq:PDCApprox2} \\
    &       &             &           $\frac{\mu}{\mu+D(f_0||f_1)}$\\
\hline
1   & 0.1   &     0.16     &   0.26\\
2   & 0.1   &     0.20     &   0.26\\
3   & 0.1   &     0.22     &   0.26\\
4   & 0.1   &     0.238     &   0.26\\
6   & 0.1   &     0.248     &   0.26\\
\hline
\end{tabular}
}
\hspace{1.5cm}
\subfloat[Fixed $D$]{\label{tab:PDCapproxFixedD}
\begin{tabular}{|l|l|l|l|l|l|l|l|l|l|l|l|l|}
\hline
&&\multicolumn{2}{c|}{$\PDC$}\\
\cline{1-4}
$D$ & $\mu$ & Simulations &   Approx      \eqref{eq:PDCApprox2} \\
    &       &             &           $\frac{\mu}{\mu+D(f_0||f_1)}$\\
\hline
6   & 0.01   &     0.033     &   0.034\\
6   & 0.05   &     0.145     &   0.151\\
6   & 0.2   &     0.37     &   0.41\\
6   & 0.3   &     0.46     &   0.51\\
6   & 0.4   &     0.51     &   0.58\\
6   & 0.6   &     0.58     &   0.68\\
\hline
\end{tabular}
}
  \caption{Comparison of $\PDC$ obtained using simulations with the approximation \eqref{eq:PDCApprox2} for $f_0 \sim {\cal N}(0,1)$, $f_1 \sim {\cal N}(0.75,1)$ and $h=\infty$. }
  \label{tab:PDCComparison}
\end{table}

\section{Conclusions and future work}
\label{sec:Conclusion}
We proposed two minimax formulations for data-efficient non-Bayesian quickest change detection,
that are extensions of the standard minimax formulations in \cite{lord-amstat-1971}
and \cite{poll-astat-1985} to the data-efficient setting.
We proposed an algorithm called the DE-CuSum algorithm, that is a modified version of the CuSum algorithm
 from \cite{page-biometrica-1954},
and showed that it is asymptotically optimal
for both the minimax formulations we proposed, as the false alarm rate goes to zero.

We discussed that, like the CuSum algorithm, the DE-CuSum algorithm
can also be seen as a sequence of SPRTs, with the difference that each SPRT is now followed
by a `sleep' time, the duration of which is a function of the accumulated log likelihood
of the observations taken in the SPRT preceding it. 
This similarity was exploited
to analyze the performance of the DE-CuSum algorithm using standard renewal theory tools,
and also to show its asymptotic optimality.
We also showed in our numerical results that the DE-CuSum algorithm has good trade-off curves and provides substantial
benefits over the approach of fractional sampling.
The techniques developed in this paper and the insights obtained can be used to study data-efficient
quickest change detection in sensor networks. See
\cite{bane-veer-ssp-2012} for some preliminary results.

\section*{APPENDIX}
\begin{proof}[Proof of Lemma~\ref{thm:lambdaDUpperBound}]
If $0 < D(f_0 \;||\; f_1)< \infty$, then $\Expect_\infty[\lambdaInf] < \infty$.
Thus, $\Prob_\infty(\lambdaInf < \infty) = 1$. Choose an arbitrary $D$, and partition $\{\lambdaInf < \infty\}$
in to three events:
\[\mathcal{A} = \{\lambdaInf < \infty\} \cap \{L(X_1) < 0\},\]
\[\mathcal{B} = \{\lambdaInf < \infty\} \cap \{L(X_1) \geq 0\} \cap \{W_n \mbox{ never crosses} \; D\},\]
\[\mathcal{C} = (\mathcal{A} \cup \mathcal{B})'.\]
Then, clearly
\[
\Prob_\infty(\mathcal{A}) = \Prob_\infty(L(X_1) < 0),
\]
and
\begin{eqnarray*}
\Prob_\infty(\mathcal{A}\cup \mathcal{B}) &=& \Prob_\infty(W_{\lambdad}<0) \\
& > & \Prob_\infty(\mathcal{A}) \\
&=& \Prob_\infty(L(X_1) < 0) > 0.
\end{eqnarray*}
Thus, $\Expect_\infty[\lambdad \Big| W_{\lambdad}<0]$ is well defined and
\begin{equation*}
\begin{split}
\Expect_\infty[\lambdaInf] &\geq \Expect_\infty[\lambdaInf ; \mathcal{A} \cup \mathcal{B}]\\
                           & \geq   \Expect_\infty[\lambdaInf \Big| \mathcal{A} \cup \mathcal{B}]\; \Prob_\infty(\mathcal{A})\\
                           & =   \Expect_\infty[\lambdad \Big| W_{\lambdad}<0] \; \Prob_\infty(L(X_1) < 0).\\
\end{split}
\end{equation*}
This proves the lemma because $\Prob_\infty(L(X_1) < 0) > 0$.
\end{proof}
\smallskip

\begin{proof}[Proof of Lemma~\ref{thm:EInfTxy_UpperLower}]
Since $T(x,y,\mu) = \lceil |y-x|/\mu \rceil$, we have
\[\frac{|y-x|}{\mu} \leq T(x,y,\mu) \leq \frac{|y-x|}{\mu} +1.  \]
We will use this simple inequality to obtain the upper and lower bounds.

We first obtain the upper bound. Clearly,
\begin{equation*}
\begin{split}
\Expect_\infty[T(W_{\lambdad}^{h+} ,0,\mu) \Big| W_{\lambdad} < 0] \leq \frac{\Expect_\infty[|W_{\lambdad}^{h+}| \; \Big| \; W_{\lambdad} < 0]}{\mu}+1.
\end{split}
\end{equation*}
An upper bound for the right hand side of the above equation is easily obtained.
First note that from \eqref{eq:EinfWInf_hplus_finite}
\[ \Expect_\infty[|W_{\lambdaInf}^{h+}|] \; \leq \; \Expect_\infty[|W_{\lambdaInf}|] \; < \; \infty. \]
Thus, from the notation introduced in the proof of Lemma~\ref{thm:lambdaDUpperBound} above
\begin{equation*}
\begin{split}
\Expect_\infty[|W_{\lambdaInf}^{h+}|] &\geq \Expect_\infty[|W_{\lambdaInf}^{h+}| ; \mathcal{A} \cup \mathcal{B}]\\
                           & \geq   \Expect_\infty[|W_{\lambdaInf}^{h+}|  \;\Big| \; \mathcal{A} \cup \mathcal{B}]\ \Prob_\infty(\mathcal{A})\\
                           & =   \Expect_\infty[|W_{\lambdad}^{h+}|  \;\Big| \; W_{\lambdad}<0] \; \Prob_\infty(L(X_1) < 0).\\
\end{split}
\end{equation*}
This completes the proof for the upper bound.

For the lower bound we have
\begin{equation*}
\begin{split}
\Expect_\infty[T(W_{\lambdad}^{h+} &,0,\mu)\; \Big| \; W_{\lambdad} < 0] \\
&\geq \frac{\Expect_\infty[|W_{\lambdad}^{h+}| \; \Big| \; W_{\lambdad} < 0]}{\mu}\\
&\geq \frac{\Expect_\infty[|W_{\lambdad}^{h+}| \; ; \{L(X_1)<0\}\; \Big| \; W_{\lambdad} < 0]}{\mu}\\
& = \frac{\Expect_\infty[|L(X_1)^{h+}| \; \Big| \; L(X_1)<0]}{\mu}  \; \Prob_\infty(L(X_1) < 0)\\
\end{split}
\end{equation*}
\end{proof}

\begin{proof}[Proof of Lemma~\ref{thm:E1Txy}]
First note that
\[\Prob_1(W_{\lambdad} < 0) > \Prob_1(L(X_1)<0) > 0.\]
Thus, $\Expect_1[T((W_{\lambdad})^{h+},0,\mu)| W_{\lambdad} < 0]$ is well defined.
Also using the inequality on $T(x,y,\mu)$ from Lemma~\ref{thm:EInfTxy_UpperLower} we have
\begin{equation}\label{eq:lemTempeq1}
\begin{split}
\Expect_1[T(W_{\lambdad}^{h+} ,0,\mu) \Big| W_{\lambdad} < 0] \leq \frac{\Expect_1[|W_{\lambdad}^{h+}| \; \Big| \; W_{\lambdad} < 0]}{\mu}+1 
\end{split}
\end{equation}

We now get an upper bound on the right hand side of the above equation. By Wald's likelihood
ratio identity \cite{sieg-seq-anal-book-1985} and \eqref{eq:EinfWInf_hplus_finite},
\begin{equation}\label{eq:lemTempeq2}
\begin{split}
\Expect_1[|W_{\lambdaInf}^{h+}| \; &; \; W_{\lambdaInf} < 0] \\
&= \Expect_1[|W_{\lambdaInf}^{h+}| \; ; \; \lambdaInf < \infty]  \\
&= \Expect_\infty[|W_{\lambdaInf}^{h+}| \prod_{k=1}^{\lambdaInf} L(X_k)\; ; \; \lambdaInf < \infty]  \\
&= \Expect_\infty[|W_{\lambdaInf}^{h+}| \; e^{W_{\lambdaInf}} \; ; \; W_{\lambdaInf} < 0]\\
&\leq \Expect_\infty[|W_{\lambdaInf}^{h+}|] \; \leq \; \Expect_\infty[|W_{\lambdaInf}|]\; <\;  \infty.\\
\end{split}
\end{equation}
Using again the notation introduced in the proof of Lemma~\ref{thm:lambdaDUpperBound} we have
\begin{equation}\label{eq:lemTempeq3}
\begin{split}
\Expect_1[|W_{\lambdaInf}^{h+}| \; &; \; W_{\lambdaInf} < 0]  \\
&\geq \Expect_1[|W_{\lambdaInf}^{h+}| \; ; \; (\{W_{\lambdaInf} < 0\}) \cap (\mathcal{A} \cup \mathcal{B})]\\
&= \Expect_1[|W_{\lambdaInf}^{h+}| \; ; \; \mathcal{A} \cup \mathcal{B}]\\
& \geq   \Expect_1[|W_{\lambdaInf}^{h+}|  \;| \; \mathcal{A} \cup \mathcal{B}]\ \Prob_1(\mathcal{A})\\
& =   \Expect_1[|W_{\lambdad}^{h+}|  \;| \; W_{\lambdad}<0] \; \Prob_1(L(X_1) < 0).\\
\end{split}
\end{equation}
Thus from \eqref{eq:lemTempeq1}, \eqref{eq:lemTempeq2}, and \eqref{eq:lemTempeq3}
\begin{equation*}
\begin{split}
\Expect_1[T(W_{\lambdad}^{h+} ,0,\mu)& | W_{\lambdad} < 0]\\
&\leq \frac{\Expect_1[|W_{\lambdad}^{h+}| \; \Big| \; W_{\lambdad} < 0]}{\mu }+1 \\
&\leq \frac{\Expect_1[|W_{\lambdaInf}^{h+}| \; ; \; W_{\lambdaInf} < 0]}{\mu \; \Prob_1(L(X_1)<0)}+1 \\
&\leq \frac{\Expect_\infty[|W_{\lambdaInf}^{h+}| ]}{\mu \; \Prob_1(L(X_1)<0)}+1 \\
& < \infty.
\end{split}
\end{equation*}
This proves the lemma.
\end{proof}

\begin{proof}[Proof of Lemma~\ref{thm:resettingLemma}]
Let
\[\tauc(x) = \inf\{n\geq 1: C_n > D; C_0=x\}.\]
Here, $C_n$ is the CuSum statistic and evolves according the description of the algorithm in
Algorithm~\ref{algo:CuSum}. Thus, $\tauc(x)$ is the first time for the CuSum algorithm
to cross $D$, when starting at $C_0=x$. Clearly, $\tauc(x) = \tauc$ if $x=0$.
It is easy to see by sample path wise arguments that
\[\Expect_1[\tauc(x)] \leq \Expect_1[\tauc].\]
The proof depends on the above inequality.

Let $\mathcal{A}_x$ be the event that the CuSum statistic, starting with $C_0=x$, touches
zero before crossing the upper threshold $D$. Let $q_x = \Prob_1(\mathcal{A}_x)$. Then,
\[\Expect_1[\tauc(x)] = \Expect_1[\tauc(x); \mathcal{A}_x] + \Expect_1[\tauc(x); \mathcal{A}_x'] \leq \Expect_1[\tauc].\]
Note that
\[\Expect_1[\tauc(x); \mathcal{A}_x'] = \Expect_1[\tauw(x); \mathcal{A}_x'].\]
We call this common constant $\mathsf{t_1}$. Also note that on $\mathcal{A}_x$,
the average time taken to reach 0 is the same for both the CuSum and the DE-CuSum algorithm.
We call this common average conditional delay by $\mathsf{t_2}$. Thus,
\[\Expect_1[\tauc(x)] = (\mathsf{t_1}) (1-q_x) + q_x (\mathsf{t_2} + \Expect_1[\tauc]) \leq \Expect_1[\tauc].\]
The equality in the above equation is true because, once the DE-CuSum statistic reaches zero, it is
reset to zero and the average delay that point onwards is $\Expect_1[\tauc]$.

Then for any $\mathsf{t_3} \geq \Expect_1[\tauc]$ we have
\[(\mathsf{t_1}) (1-q_x) + q_x (\mathsf{t_2} + \mathsf{t_3}) \leq \mathsf{t_3}.\]
This is because for $\mathsf{t_3} \geq \Expect_1[\tauc]$
\begin{eqnarray*}
(\mathsf{t_1}) (1-q_x) &+& q_x (\mathsf{t_2} + \mathsf{t_3}) \\
&=& (\mathsf{t_1}) (1-q_x) + q_x (\mathsf{t_2} + \Expect_1[\tauc] + \mathsf{t_3} - \Expect_1[\tauc])\\
&\leq &  \Expect_1[\tauc] + q_x ( \mathsf{t_3} - \Expect_1[\tauc]) \\
& \leq & \mathsf{t_3}.
\end{eqnarray*}
It is easy to see that
\[\Expect_1[\tauw(x)] \leq (\mathsf{t_1}) (1-q_x) + q_x (\mathsf{t_2} + T_U^{(1)}(h,\mu) + \Expect_1[\tauw]).\]
This is because on $\mathcal{A}_x$, the average delay of the DE-CuSum algorithm is
the average time to reach 0, which is $\mathsf{t_2}$,
 plus the average time spent below 0 due to the undershoot, which is bounded from above by
 $T_U^{(1)}(h,\mu)$, plus the average delay after the sojourn below 0, which is $\Expect_1[\tauw]$.
The latter is due to the renewal nature of the DE-CuSum algorithm.
Since $T_U^{(1)}(h,\mu) + \Expect_1[\tauw] \geq \Expect_1[\tauc]$, the first part of lemma is
proved if we set $\mathsf{t_3} = T_U^{(1)}(h,\mu) + \Expect_1[\tauw]$.

For the second part, note that $T_U^{(1)}(h,\mu) \leq \lceil h/\mu \rceil$.
\end{proof}

\bibliographystyle{ieeetr}


\bibliography{QCD_verSubmitted}

\begin{thebibliography}{10}

\bibitem{poor-hadj-qcd-book-2009}
H.~V. Poor and O.~Hadjiliadis, {\em Quickest detection}.
\newblock Cambridge University Press, 2009.

\bibitem{tart-niki-bass-2013}
A.~G. Tartakovsky, I.~V. Nikiforov, and M.~Basseville, {\em Sequential
  Analysis: {Hypothesis} Testing and Change-Point Detection}.
\newblock Statistics, CRC Press, 2013.

\bibitem{veer-bane-elsevierbook-2013}
V.~V. Veeravalli and T.~Banerjee, {\em Quickest Change Detection}.
\newblock Elsevier: E-reference Signal Processing, 2013.
\newblock \url{http://arxiv.org/abs/1210.5552}.

\bibitem{shir-siamtpa-1963}
A.~N. Shiryaev, ``On optimum methods in quickest detection problems,'' {\em
  Theory of Prob and App.}, vol.~8, pp.~22--46, 1963.

\bibitem{tart-veer-siamtpa-2005}
A.~G. Tartakovsky and V.~V. Veeravalli, ``General asymptotic {Bayesian} theory
  of quickest change detection,'' {\em SIAM Theory of Prob. and App.}, vol.~49,
  pp.~458--497, Sept. 2005.

\bibitem{lord-amstat-1971}
G.~Lorden, ``Procedures for reacting to a change in distribution,'' {\em Ann.
  Math. Statist.}, vol.~42, pp.~1897--1908, Dec. 1971.

\bibitem{poll-astat-1985}
M.~Pollak, ``Optimal detection of a change in distribution,'' {\em Ann.
  Statist.}, vol.~13, pp.~206--227, Mar. 1985.

\bibitem{mous-astat-1986}
G.~V. Moustakides, ``Optimal stopping times for detecting changes in
  distributions,'' {\em Ann. Statist.}, vol.~14, pp.~1379--1387, Dec. 1986.

\bibitem{ritov-astat-1990}
Y.~Ritov, ``Decision theoretic optimality of the {CUSUM} procedure,'' {\em Ann.
  Statist.}, vol.~18, pp.~1464--1469, Nov. 1990.

\bibitem{lai-ieeetit-1998}
T.~L. Lai, ``Information bounds and quick detection of parameter changes in
  stochastic systems,'' {\em IEEE Trans. Inf. Theory}, vol.~44, pp.~2917
  --2929, Nov. 1998.

\bibitem{tart-poll-polu-arxiv-2011}
A.~G. Tartakovsky, M.~Pollak, and A.~Polunchenko, ``Third-order asymptotic
  optimality of the generalized {Shiryaev-Roberts} changepoint detection
  procedures,'' {\em ArXiv e-prints}, May 2010.

\bibitem{robe-technometrics-1966}
S.~W. Roberts, ``A comparison of some control chart procedures,'' {\em
  Technometrics}, vol.~8, pp.~411--430, Aug. 1966.

\bibitem{page-biometrica-1954}
E.~S. Page, ``Continuous inspection schemes,'' {\em Biometrika}, vol.~41,
  pp.~100--115, June 1954.

\bibitem{bane-veer-sqa-2012}
T.~Banerjee and V.~V. Veeravalli, ``Data-efficient quickest change detection
  with on-off observation control,'' {\em Sequential Analysis}, vol.~31,
  pp.~40--77, Feb. 2012.

\bibitem{taga-jqt-1998}
G.~Tagaras, ``A survey of recent developments in the design of adaptive control
  charts,'' {\em Journal of Quality Technology}, vol.~30, pp.~212--231, July
  1998.

\bibitem{stou-etal-jamstaa-2000}
Z.~G. Stoumbos, M.~R. Reynolds, T.~P. Ryan, and W.~H. Woodall, ``The state of
  statistical process control as we proceed into the 21st century,'' {\em J.
  Amer. Statist. Assoc.}, vol.~95, pp.~992--998, Sept. 2000.

\bibitem{prem-kuma-infocom-2008}
K.~Premkumar and A.~Kumar, ``Optimal sleep-wake scheduling for quickest
  intrusion detection using wireless sensor networks,'' in {\em IEEE Conference
  on Computer Communications (INFOCOM)}, pp.~1400--1408, Apr. 2008.

\bibitem{geng-etal-samplingrights-allerton-2012}
J.~Geng, L.~Lai, and E.~Bayraktar, ``Quickest change point detection with
  sampling right constraints,'' in {\em Annual Allerton Conference on
  Communication, Control, and Computing (Allerton)}, Sept. 2012.

\bibitem{kris-ieeetit-whento-2012}
V.~Krishnamurthy, ``When to look at a noisy markov chain in sequential decision
  making if measurements are costly?,'' {\em ArXiv e-prints}, Aug. 2012.

\bibitem{sieg-seq-anal-book-1985}
D.~Siegmund, {\em Sequential Analysis: {Tests} and Confidence Intervals}.
\newblock Springer series in statistics, Springer-Verlag, 1985.

\bibitem{wald-wolf-amstat-1948}
A.~Wald and J.~Wolfowitz, ``Optimum character of the sequential probability
  ratio test,'' {\em Ann. Math. Statist.}, vol.~19, no.~3, pp.~pp. 326--339,
  1948.

\bibitem{wood-nonlin-ren-th-book-1982}
M.~Woodroofe, {\em Nonlinear Renewal Theory in Sequential Analysis}.
\newblock CBMS-NSF regional conference series in applied mathematics, SIAM,
  1982.

\bibitem{bane-veer-ssp-2012}
T.~Banerjee and V.~V. Veeravalli, ``Energy-efficient quickest change detection
  in sensor networks,'' in {\em IEEE Statistical Signal Processing Workshop},
  Aug. 2012.

\end{thebibliography}
\end{document}